%% file: GroupInf-Revision-20160218.tex
\documentclass[11pt, ejs, preprint]{imsart}
\usepackage[utf8]{inputenc}
\RequirePackage[OT1]{fontenc}
\usepackage[margin=1.2in]{geometry}
\usepackage[dvips]{graphics}
\DeclareGraphicsExtensions{.eps.gz,.eps,.epsi.gz,.epsi,.ps,.ps.gz}
\usepackage{epsfig}
\usepackage{color}
\usepackage{latexsym}
\usepackage{graphicx}
\usepackage{amsfonts}
\usepackage{amssymb}
\usepackage{mathrsfs}
\usepackage{verbatim}
\usepackage{amsmath}
\usepackage{amsthm}
\usepackage{latexsym}
\usepackage{mathrsfs}
\usepackage{amsfonts}
\usepackage{color}
\usepackage{import}
\usepackage{float}
\usepackage[titletoc]{appendix}
\usepackage{array,booktabs}
\usepackage{multirow}
\usepackage[pagebackref=true]{hyperref} 
\hypersetup{
    unicode=false,          
    pdffitwindow=false,     
    pdfstartview={FitH},    
    pdftitle={My title},    
    pdfauthor={Author},     
    pdfsubject={Subject},   
    pdfcreator={Creator},   
    pdfproducer={Producer}, 
    pdfkeywords={keyword1} {key2} {key3}, 
    pdfnewwindow=true,      
    colorlinks=true,       
    linkcolor=blue,          
    citecolor=blue,        
    urlcolor=blue       
}

\usepackage{titlesec}
\titleformat{\section}[runin]  
  {\normalfont\bfseries}{\thesection.}{0.4em}{\addperiod}
\titleformat{\subsection}[runin] 
  {\normalfont\bfseries}{\thesubsection.}{0.4em}{\addperiod}
\newcommand{\addperiod}[1]{#1.}
  
\RequirePackage[authoryear]{natbib}
\usepackage[nottoc]{tocbibind}

\startlocaldefs
\numberwithin{equation}{section}
\theoremstyle{plain}
\newtheoremstyle{mystyle} 
    {\topsep}                    
    {\topsep}                    
    {}                   
    {}                           
    {\bfseries}                   
    {.}                          
    {.5em}                       
    {}  
\newtheorem*{theorem*}{Theorem}
\newtheorem{theorem}{Theorem}
\newtheorem{lemma}{Lemma}
\newtheorem{proposition}{Proposition}
\newtheorem{corollary}{Corollary}

\theoremstyle{remark}
\newtheorem{remark}{Remark}
\theoremstyle{mystyle}

 \def\bbetastar{\bbeta^{*}}\def\sigstar{{\sigma^{*}}}
 \def\Sstar{S^{*}}
\def\bugj{\bu_{\Gj}}

\def\bbetagj{\bbeta_{\Gj}}

\def\hbbetagj{\hbbeta_{\Gj}}
\def\bbetagstar{\bbeta_{G}^{*}}\def\hbbetag{\hbbeta_{G}}
\def\matXgj{\matX_{\Gj}}
\def\matXg{\matX_{G}}
\def\matZg{\matZ_{G}}
 
\def\Gj{G_{j}}\def\Gk{G_{k}}\def\dj{d_{j}}

\def\bmugj{\bmu_{\Gj}}

\def\SCIF{{\rm SCIF}}

\def\eps{\epsilon}

\endlocaldefs

\input{riikdef.tex}

\begin{document}

\begin{frontmatter}
\title{{The Benefit of Group Sparsity in Group Inference with De-biased Scaled Group Lasso}}
\runtitle{Group Inference}
\begin{aug}
  \author{\fnms{ Ritwik }  \snm{Mitra}\corref{}\thanksref{}\ead[label=e1]{rmitra@princeton.edu}}
    \address{ORFE, Princeton University, NJ 08540}
    \printead{e1}\\
  \and
  \author{\fnms{Cun-Hui } \snm{Zhang}\thanksref{t4}\ead[label=e2]{czhang@stat.rutgers.edu}}
  \address{Dept. of Statistics \& Biostatistics, Rutgers University, NJ 08854}
    \printead{e2}

\thankstext{t4}{Research was supported in part by the NSF Grants
DMS-11-06753 and DMS-12-09014.}
 \runauthor{R.Mitra \& C.-H. Zhang}



\end{aug}


\begin{abstract}
We study confidence regions and approximate chi-squared tests for variable groups 
in high-dimensional linear regression. 
When the size of the group is small, low-dimensional projection estimators for individual 
coefficients can be directly used to construct efficient confidence regions and p-values 
for the group. 
However, the existing analyses of low-dimensional projection estimators do not directly carry 
through for chi-squared-based inference of a large group of variables without inflating the 
sample size by a factor of the group size. 
We propose to de-bias a scaled group Lasso for chi-squared-based statistical inference for 
potentially very large groups of variables. 
We prove that the proposed methods capture the benefit of group sparsity under proper 
conditions, for statistical inference of the noise level and variable groups, large and small. 
Such benefit is especially strong when the group size is large. 

\end{abstract}



\end{frontmatter}




\section{Introduction}
We consider the linear regression model
\begin{align}\label{eq:mod1}
\by = \matX\bbetastar + \bepsa,
\end{align}
where $\matX = (\bx_1,\ldots,\bx_p) \in \Re^{n\times p}$ is a design matrix, $\by \in \Re^{n}$ is a response vector, 
$\bepsa \sim \sfN_{n}(\bzero,\sigma^{2}\matI_{n})$ with an unknown noise level $\sigma$, and $\bbetastar = (\beta^{*}_1,\ldots,\beta^{*}_p)^T \in \Re^{p}$ is the  vector of unknown true regression coefficients. 
We are interested in making statistical inference about a group of coefficients 
$\bbetastar_G=(\beta^{*}_j, j\in G)^T$. For small $p$, the $F$-distribution, which is approximately chi-squared with proper normalization, provides classical confidence regions for $\bbetastar_{G}$ and p-values for testing 
$\bbetastar_{G}$. We want to construct approximate versions of such procedures for potentially very large groups in high-dimensional models where $p$ is large, possibly much larger than $n$. 

The study of asymptotic inference for parameter estimates in high dimensional regression has experienced a flurry of research activities in recent years. Many attempts have been made to assess the model selected by high dimensional regularizers; for example, some early work was done in \cite{KnightFu00}, 
sample splitting was considered in \cite{WasRoe09} and \cite{mein09}, 
and subsampling was considered in  \cite{mein10} and \cite{shah13}.  See \cite{buhlsara11} for more detailed account of some of these methods.
\cite{LeebP06} proved that the sampling distribution of statistics based on selected models is not estimable. \cite{BerkBZ10} proposed 
conservative approaches. Alternative approaches were proposed in \cite{Lockhart14} and \cite{Mein13}. 

Recent works in \cite{Zhang2014}, \cite{VandeGeer2013} 
and \cite{Javanmard2013b} among others are more relevant to the line of research 
we have adopted in the current work, which we describe in some detail. 
For the effect of a preconceived variable, 
\cite{Zhang2014} pointed out the feasibility of regular statistical inference at the parametric $n^{-1/2}$ 
rate by correcting the bias of a regularized estimator of the entire coefficient vector, such as the Lasso, 
and proposed a low-dimensional projection estimator (LDPE) to carry out the task. 
The basic idea is to project the residual of the regularized estimator to the direction of a certain 
score vector which is approximately orthogonal to all variables other than the preconceived one. 
Such bias correction, which has been called de-biasing, is parallel to 
correcting the bias of nonparametric estimators in semiparametric inference \citep{BKRW93}. 
In a general setting, \cite{ZhangOber11} developed an alternative formulation of the LDPE 
and provided formulas for the direction of the least favorable submodel 
and the Fisher information bound for the asymptotic variance. 
In linear regression, the least favorable submodel more explicitly connects 
the Lasso estimator of the score vector to column-by-column estimation of 
the precision matrix for random designs \citep{CaiLL11, SunZ13}. 
\cite{Buhl13} developed and studied methods to correct the bias of ridge regression. 
\cite{Belloni14} considered estimation of treatment effects with a large number of controls. 
\cite{VandeGeer2013} proved that the LDPE attains the Fisher information bound 
under a sparsity condition on the precision matrix and 
made a connection between the Lasso estimation of the score vector and 
the inversion of the Karush-Kuhn-Tucker (KKT) conditions through the precision matrix. 
Moreover, \cite{VandeGeer2013} extended their results to generalized linear models (GLMs) with an innovative way of analyzing such models.
\cite{Javanmard2013b} proved that when a {}{quadratic} programming method of \cite{Zhang2014} 
is used {}{to} estimate the score vector, the LDPE attains the Fisher information bound 
for Gaussian designs without requiring sparsity condition on the precision matrix; 
see Subsection \ref{subsec:biascorrect} for further discussion. 

In a separate work, \cite{Javanmard2013a} considered inference with lower sample size requirements 
when the design is known to be standard Gaussian. 
\cite{SunZ12PartialCorr}, \cite{RSZZ12} and \cite{Jankova14} considered extensions to 
graphical models and precision matrix estimation. 

It is possible to directly extend the above described de-biasing method to the case of grouped variables. 
In fact, the LDPE provides 
\bel{expansion}
\sqrt{n}\big(\hbbeta_G - \bbetastar_G\big)  = \sfN_{|G|}\big(\bzero, \sigma^2 \matV_{G,G}\big) + \Rem_G
\eel
along with a known covariance structure $\matV_{G,G}$ and 
$\|\Rem_G\|_\infty\lesssim \|\bbeta^*\|_0(\log p)/\sqrt{n}$ \citep{Zhang2014}. 
However, this does not directly provide a sharp error bound for 
the $\ell_2$- or equivalently chi-squared-based group inference for large groups. 
As $\Var(\chi_{|G|}) \approx 1/2$, the trivial bound 
$\|\Rem_G\|_2 \lesssim |G|^{1/2}\|\bbeta^*\|_0(\log p)/\sqrt{n}=o(1)$ for group inference  
leads to an extra factor $|G|$ in the sample size requirement. 
Thus, the group inference problem is unsolved when one is unwilling to impose such a strong condition 
on $n$. 
Our goal is to construct $\hbbeta_G$ satisfying $\|\Rem_G\|_2 = o(1)$ in an 
expansion of the form (\ref{expansion}) with potentially very large $|G|$. The impact of such a result is certainly beyond the specific problem under consideration. 

Our approach is based on the natural idea that group sparsity can be exploited 
in statistical inference of variable groups. 
To this end, we propose to use a linear estimator to correct the bias of a scaled group Lasso estimator.  This combines and extends the ideas of the group Lasso \citep{Yuan2006} and bias correction \citep{Zhang2014}, and will be shown to capture the benefit of group sparsity in both high-dimensional estimation as in \cite{HZ10} and in bias correction. We note that the type of statistical inference under consideration here is regular in the sense that it does not require model selection consistency, and that it attains asymptotic efficiency in the sense of Fisher information without being super-efficient. 
A characterization of such inference is that it does not require a uniform signal strength condition on informative features, e.g. a lower bound on the non-zero $|\beta_j|$ above an inflated noise level due to model uncertainly, known as the ``beta-min'' condition.

Since our proposed method relies upon a group regularized initial estimator, in the following we provide a 
brief discussion of the literature on the topic. The group Lasso \citep{Yuan2006} can be defined as
\begin{align}\label{eq:grplsoopt-ns}
\hbbeta(\bomega) = \argmin_{\bbeta} \calL_{\bomega}(\bbeta), \quad  \calL_{\bomega}(\bbeta) = \dfrac{\norm{\by-\matX\bbeta}{2}^{2}}{2n} + \sum^{M}_{j=1}\omega_{j}\norm{\bbetagj}{2}, 
\end{align}
where $\{G_j,, 1\le j\le M\}$ forms a partition of the index set $\{1,\ldots,p\}$ of variables. 
It is worthwhile to note that when the group effects are being regularized, 
the choice of the basis $\matX_{G_j}=(\bx_k, k\in G_j)$ within the group may not 
play a prominent role, so that the design is often ``pre-normalized'' to 
satisfy $\matX_{G_j}^T\matX_{G_j}/n = \matI_{G_j\times G_j}$ as in \cite{Yuan2006}. 
The group Lasso and its variants have been studied in \cite{Bach08}, \cite{Kol08}, \cite{obo08}, \cite{NR08}, \cite{LZ09}, \cite{HZ10}, and \cite{Lounici2011} 
among many others. 
\cite{HZ10} characterized the benefit of group Lasso in $\ell_{2}$ estimation, versus the Lasso \citep{Tib96}, 
under the assumption of \emph{strong group sparsity}; 
see (\ref{def:strnggrpsprse}) in Section \ref{sec:groupinf}. 
\cite{HuangMXZ09} and \cite{BrehenyH11} developed methodologies for concave group and bi-level regularization. We refer to \cite{HuangBM12} for further discussion and additional references

Estimation of the scale parameter, or the noise level $\sigma$, 
is also an important aspect of high dimensional regularized regression. 
Due to scale invariance, it is natural to let the groupwise weights in (\ref{eq:grplsoopt-ns}) be 
proportional to the scale parameter $\sigma$. Thus, a consistent estimate of $\sigma$ also becomes 
necessary for truly adaptive estimation of the parameters. 
For the Lasso problem, \cite{Antoniacomm10} and \cite{SZcomm10, Sun2012} 
proposed a scaled Lasso that estimates both the scale parameter $\sigma$ and 
coefficient vector $\bbetastar$, which is closely related to the earlier proposals 
of \cite{ZhangMCP10} and \cite{SBG10}. This scaled Lasso turns out to be equivalent to treating the residual of \citeauthor{BelloniCW11}'s (\citeyear{BelloniCW11}) square-root Lasso estimator of $\bbeta$ as the noise vector in the estimation of of $\sigma$. For group regularization, \cite{BLS14} proposed a square-root group Lasso for adaptive estimation of 
the coefficient vector $\bbeta$. 
In this paper, we study a scaled group Lasso for simultaneous estimation of 
both $\bbeta$ and $\sigma$ with a different weighted $\ell_{2,1}$ penalty 
and prove the benefit of grouping in the estimation of the scale parameter 
in terms of convergence rates. 

This paper is organized as follows. 
In Section \ref{sec:groupinf}, we describe a general procedure for statistical inference of 
groups of variables and provide theoretical guarantees for our results. 
In Section \ref{sec:consisres}, we study the scaled group Lasso needed for the 
construction of estimators in Section~\ref{sec:groupinf}. 
In Section \ref{sec:simures}, we present some simulation results to demonstrate the feasibility and 
performance of the proposed methods. 
In Section \ref{sec:summdis} we provide a brief summary of our results 
and discuss future directions of research. 
Proofs of some technical results are relegated to the \nameref{sec:appendix}.


\smallskip
We use the following notation throughout the paper. 
For vectors $\bu \in \Re^{d}$, the $\ell_{p}$ norm is denoted by 
$\norm{\bu}{p} =(\sum^{d}_{k=1}|u_{k}|^{p})^{1/p}$,  
with $\|\bu\|_\infty = \max_{1\le k\le d}|u_k|$ and $\norm{\bu}{0}= \#\{j:u_{j}\neq 0\}$. 
For matrices $\matA$, 
the Moore-Penrose pseudo inverse is denoted by $\matA^\dag$, 
the spectrum norm is denoted by $\|\matA\|_S = \max_{\|\bu\|_2=\|\bv\|_2=1}\bu^T\matA\bv$, 
the Frobenius norm by $\|\matA\|_F =\{\trace(\matA^T\matA)\}^{1/2}$, and 
the nuclear norm by $\|\matA\|_N =\max_{\|\matB\|_S=1}\trace(\matB^T\matA)$. Given $A\subset \{1,\cdots, p\}$, for any vector $\bu\in \Re^{p}$, $\bu_{A} \in \Re^{|A|}$ denotes a vector with corresponding components from $\bu$, $\matX_{A} \in \Re^{n\times |A|}$ denotes the sub-matrix of $\matX$ with corresponding columns as indicated by the set $A$, $\matX_{-A}$ denotes the sub-matrix of $\matX$ with column indices belonging to the complement of $A$, $\calR(\matX_{A})$ denotes the column space spanned by columns of 
$\matX_{A}$, $\matQ_A = \matX_A(\matX_A^{T}\matX_A)^{\dag}\matX_A^{T}$ denotes the orthogonal 
projection to $\calR(\matX_{A})$, and $\matQ_{A}^{\perp} = \matI_{p\times p} - \matQ_A$. 
Additionally, $\bbE$ and $\bbP$ denote respectively the expectation and probability measure.

\section{Group Inference} \label{sec:groupinf}
We present our results in seven subsections. 
Subsection \ref{subsec:grpstructure} describes the group structure of the regression problem in detail and the notion of strong group sparsity. 
Subsection \ref{subsec:biascorrect} provides a brief account 
of the bias correction procedure for statistical inference of a single variable. 
Subsection \ref{subsec:biascorrect-group} proposes an extension of the bias correction idea to group inference. 
Subsection~\ref{subsec:ideal-sol} justifies the proposed group inference methodology in an ideal setting 
and states a working assumption for more general settings. 
Subsection~\ref{subsec:optstrategy} provides optimization methods 
for construction of group inference procedures 
under the working assumption. 
Subsection \ref{subsec: Feasibility} provides sufficient conditions 
for the feasibility of the optimization scheme considered in Subsection \ref{subsec:optstrategy}. 
Subsection \ref{subsec:findSolution} discusses convexations of the optimization problem 
and summarizes the overall scheme.

\subsection{Group structure and strong group sparsity}\label{subsec:grpstructure}
We assume an inherent and pre-specified non overlapping group structure of the feature set. 
Put precisely, assume that $\{1,\cdots, p\} = \cup^{M}_{j=1}\Gj$ such that $\Gj\cap\Gk = \varnothing$. 
Define $\dj=|\Gj|$ for all $j$ so that $\sum^{M}_{j=1}\dj=p$. 
For any index set $T \subset \{1,\cdots, M\}$, we define $G_{T} = \cup_{j \in T}\Gj$. 
In the following, we allow the quantities $n, p, M, \dj$'s etc. to all grow to infinity. 
 
In light of this group structure, further results on consistency of group regularized estimators of 
$\bbetastar$ will be based on a weighted mixed $\ell_{2,1}$ norm, defined as 
$\sum^{M}_{j=1}\omega_{j}\norm{\bugj}{2}$ for $\bu =(\bugj;1\leq j\leq M)\in \Re^{p}$ with $\bugj \in \Re^{|\Gj|}$, 
where $\bomega=(\omega_{1},\cdots, \omega_{M})\in \Re^{M}$ with $\omega_{j}>0$ for all $j$.  
This norm will be used both as penalty and as a key loss function. 
Weighted mixture norm of this type provides suitable description of the 
complexity of the unknown $\bbeta$ when the following strong group sparsity condition of \cite{HZ10}  holds. 

\smallskip{\bf Strong group sparsity:} 
{\it With the given group structure $\{G_j, j=1,\ldots,M\}$ as a partition of $\{1,\ldots,p\}$, 
there exists a group-index set, $S^{*}\subset\{1,\cdots,M\}$, such that
\begin{align}\label{def:strnggrpsprse}
\quad |S^{*}| \leq g, \quad |G_{S^{*}}|\leq s, \quad \supp(\bbetastar) \subset G_{S^{*}}=\cup_{j\in S^*}G_j.
\end{align}
In this case, we say that the true coefficient vector $\bbetastar$ is $(g,s)$ strongly group sparse 
with group support $S^*$.}
\smallskip

Our aim is to make chi-squared-type statistical inference about the effect of a group $G$ of variables, 
including confidence regions and $p$-values for $\matXg\bbetagstar$ and $\bbetagstar$. 
As will be clear from our analysis, the methodologies proposed in this paper will allow the size of the 
group $G$ to grow unboundedly up to $|G|=o(n)$. 
Moreover, the group $G$ of interest does not have to be congruent with the group structure 
$\{G_j, j=1,\ldots,M\}$. 
In fact, each of the $|G|$ variables in $G$ could belong to any of the  $M$ different pre-specified groups of variables so that 
\[
\matXg\bbetagstar = \sum_{k: \Gk\cap G \neq \emptyset} \matX_{\Gk \cap G} \bbetastar_{\Gk \cap G}.
\]
Thus we can rewrite the regression problem (\ref{eq:mod1}) as 
\begin{align}\label{eq:mod2}
\by =\matX_G\bbetastar_G 
+ \sum_{\Gk\not\subseteq G} \matX_{G_k\setminus G}\bbetastar_{G_k\setminus G} + \bepsa 
= \bmu^{*}_G + \sum_{\Gk\not\subseteq G} \bmu^{*}_{\Gk\setminus G} + \bepsa,
\end{align}
where for any $A\subset\{1,\cdots,p\}$, $\bmu^{*}_{A}=\matX_{A}\bbetastar_{A}$. In the simplest case, when the variable group of interest $G$ matches the group structure in the sense that,
\bel{nested}
\matX_G\bbeta^*_G = \sum_{G_k\cap G\neq\emptyset}\matX_{G_k}\bbeta^*_{G_k}, 
\eel
(e.g. $G=G_{j_{0}}$ for some $1\leq j_{0}\leq M$), (\ref{eq:mod2}) could be simplified as,
\begin{align*}
\by =\matX_G\bbetastar_G 
+  \sum_{G_k\cap G = \emptyset} \matX_{\Gk}\bbetastar_{\Gk} + \bepsa 
= \bmu^{*}_G + \sum_{G_k\cap G = \emptyset}\bmu^{*}_{\Gk} + \bepsa.
\end{align*}

\subsection{Bias correction for a single coefficient}\label{subsec:biascorrect} 
In high-dimensional regression, regularized estimators have been extensively studied 
and proven to be consistent for the estimation of the entire mean vector $\matX\bbeta$ 
and coefficient vector $\bbeta$ under various loss functions.  
However, since such estimators are typically nonlinear 
and biased, their sampling distribution is typically intractable. 
\cite{Zhang2014} proposed to correct the bias of a regularized estimator $\hbbeta^{(init)}$ 
with an LDPE of the following form: 
\bel{LDPE-j}
\hbeta_j = \hbeta_j^{(init)} + \bz_j^T\big(\by -\matX\hbbeta^{(init)}\big)\big/\bz_j^T\bx_j, 
\eel
where $\bz_j$ is a certain score vector depending on $\matX$ only.  
Here we provide a brief review of some ideas involved in this methodology 
to prepare their extension to group inference. 

The basic idea of the LDPE can be briefly explained as follows. 
In the low-dimensional regime where $\rank(\matX)=p\le n$, 
we may pick $\bz_j = \bx_j^\perp$ as the projection of $\bx_j$ 
to the orthogonal complement of the column space of $\matX_{-j}=(\bx_k, k\neq j)$, 
i.e. $\bz_j^T\matX_{-j}={\bf 0}$ and $\bz_j^T\bx_j= \|\bz_j^\perp\|_2^2>0$. 
For this choice $\bz_j = \bx_j^\perp$,
the $\hbeta_j$ in (\ref{LDPE-j}) is identical to the least squares estimator 
$(\bx_j^\perp)^T\by/(\bx_j^\perp)^T\bx_j$, 
and thus is unbiased regardless of the choice of the initial estimator. 
In the high dimensional case where $p > n$, $\bx_j^\perp$ is no longer a valid choice 
of $\bz_j$ as the condition $\bz_j^T\matX_{-j}={\bf 0}$ forces $\bz_j = {\bf 0}$ 
when $\matX$ is in general position. 
When $\bz_j^T\matX_{-j}\neq {\bf 0}$, the linear estimator 
$\hbeta_j^{(lin)} = \bz_j^T\by/\bz_j^T\bx_j$ has unbounded bias for the estimation of $\beta_j$ 
even if we assume the sparsity condition $\|\bbeta\|_0=1$. 
However, the linear estimator is used in (\ref{LDPE-j}) to project the residual 
$\by -\matX\hbbeta^{(init)}$ 
to the direction of $\bz_j$ for the purpose of bias correction, 
and the full strength of the unbiasedness property 
$\bz_j^T\matX_{-j}={\bf 0}$ is not necessary to reduce the bias of $\hbbeta^{(init)}$ 
to an acceptable level. 

The performance of a score vector $\bz_j$ can be 
measured by a bias factor $\eta_j$ and a noise factor $\tau_j$ defined as follows,
\bes
\eta_j = \|\bz_j^T\matX_{-j}\|_\infty/\|\bz_j\|_2,\quad \tau_j = \|\bz_j\|_2/|\bz_j^T\bx_j|. 
\ees
This can be seen from the following error decomposition for the LDPE in (\ref{LDPE-j}), 
\bel{lasso:decomp}
\hbeta_j - \beta_j = \bz_j^T\bepsa/\bz_j^T\bx_j+\tau_j \Rem_j, 
\eel
in which $\bz_j^T\bepsa/\bz_j^T\bx_j\sim N(0,\tau_j^2\sigma^2)$ and an $\ell_\infty$-$\ell_1$ split leads to 
\bel{lasso:break} 
\big|\Rem_j\big| = \big|\bz_j^T\matX_{-j}(\hbbeta^{(init)}-\bbeta^*)_{-j}\big|\big/\|\bz_j\|_2 
\le \eta_j\big\|\hbbeta^{(init)}-\bbeta^*\big\|_1. 
\eel
Thus, when $|\Rem_j|=o_{\bbP}(1)$, 
statistical inference for $\beta_j$ can be carried out 
with a consistent estimate of $\sigma$. 
For example, when $\eta_j \lesssim \sqrt{\log p}$ and 
$\|\hbbeta^{(init)}-\bbeta^*\big\|_1\lesssim \|\bbeta^*\|_0\sqrt{(\log p)/n}$, 
\bes
n\gg (\|\bbeta\|_0\log p)^2\ \Rightarrow\ 
(\hbeta_{j} - \beta_{j}^{*})/(\hsigma \tau_j)
\approx (\hbeta_{j} - \beta_{j}^{*})/(\sigma \tau_j) 
\approx \sfN(0, 1)
\ees

It is worthwhile to mention here that $\tau_j$ and $\eta_j$ are both explicitly available given $\bz_j$, 
so that the validity of the above scheme requires no stronger assumptions than an $\ell_1$ error 
bound for the estimation of $\bbeta$ and a consistent estimate of $\sigma$. 
A scaled Lasso estimator can be used as $\{\hbbeta^{(init)}, \hsigma\}$, which satisfies
\bel{lasso-error}
\left|\frac{\hsigma}{\sigma^*}-1\right| + \Big(\frac{\log p}{n}\Big)^{1/2}\|\hbbeta^{(init)}-\bbeta^*\|_1
= \calO_{\bbP}\left(\dfrac{\|\bbeta^*\|_0\log p}{n}\right), 
\eel
with $\sigstar= \norm{\by-\matX\bbetastar}{2}/\sqrt{n}$ and $s=\|\bbeta^*\|_0$ \citep{Sun2012}, 
provided an $\ell_1$ restricted eigenvalue or compatibility condition on the design \citep{BRT09, VDG09}. 
Thus, the remaining issue is to 
{}{find} a score vector $\bz_j$ with sufficiently small a bias factor $\eta_j$ and a noise factor $\tau_j$. 


For random designs with an invertible population Gram matrix $\bSigma = \bbE(\matX^T\matX/n)$, 
\cite{ZhangOber11} provided the direction of the least favorable submodel $\bbeta = \beta_j\bu$ as 
\bes
\bu^o_j = \bSigma^{-1}\bfe_j\big/\big(\bSigma^{-1})_{j,j} 
= \argmin_{\bu} \Big\{ \bu^T\bSigma\bu: \bfe_j^T\bu=1\Big\},
\ees
with $\bfe_j$ being the $j$-th canonical unit vector, and defined an ideal, efficient $\bz_j$ as  
\bes
\bz^o_j = \matX \bu^o_j. 
\ees
As the $j$-th element of $\bu^o_j$ equals 1, this can be written as a linear regression model 
\bel{z_j^o}
\bx_j = \matX_{-j}\bgamma_{-j}+\bz_j^o
\eel
with $\bgamma_{-j} = (\gamma_{1,j}, \cdots, \gamma_{j-1,-j}, \gamma_{j+1,j}, \cdots,\gamma_{p,j})^{T} 
= (- \bu^o_j)_{-j} \in \Re^{p-1}$. 

Given a design matrix $\matX$, \cite{Zhang2014} proposed two choices of $\bz_j$ 
for the LDPE in (\ref{LDPE-j}). 
The first proposal of $\bz_j$ takes a point in the Lasso path in the linear regression 
of $\bx_j$ against $\matX_{-j}$: 
\begin{align}\label{eq:lassoprojres}
\bz_{j} = \bx_{j}- \matX_{-j}\hgamma_{-j}, \quad 
\hgamma_{-j} = \argmin_{\bb}\Big\{ \norm{\bx_{j}-\matX_{-j}\bb}{2}^{2}/2n + \lambda_{j}\norm{\bb}{1}\Big\}.
\end{align}
For $p\le n$, we may take $\lambda_{j}=0$, so that $\bz_{j}= \bx^{\perp}_{j}$ and 
the $\hbeta_j$ in (\ref{LDPE-j}) is the least squares estimator of $\beta_j$. 
For $p>n$, (\ref{eq:lassoprojres}) provides a relaxed projection of $\bx_j$ via the Lasso, 
and the KKT conditions for $\bz_j$ automatically provides  
\bes
\tau_j \le 1/\|\bz_j\|_2,\quad 
\eta_j = \|\bz_j^T\matX_{-j}\|_\infty/\|\bz_j\|_2 =  n\lam_j/\|\bz_j\|_2, 
\ees
which implies $\eta_j=\sqrt{2\log p}$ with a scaled $\lam_j$ satisfying 
$\lam_j = \sqrt{\|\bz_j\|_2^2(2\log p)/n^2}$. 

The second proposal of $\bz_{j}$, 
closely related to the first one in (\ref{eq:lassoprojres}) and 
given in the discussion section of \cite{Zhang2014}, 
was a constrained variance minimization scheme 
\bel{opt-z_j}
\bz_j = \argmin_{\bz}\Big\{\|\bz\|_2^2: |\bz^T\bx_j/n|=1, \|\bz^T\matX_{-j}/n\|_\infty \le \lam_j'\Big\}. 
\eel
This quadratic program, which provides $\tau_j = \|\bz_j\|_2/n$, can be understood as 
\bes
\hbox{minimize }\ \tau_j^2\ \hbox{ subject to }\ \eta_j \le \lam_j'/\tau_j \approx \sqrt{2\log p}. 
\ees
A variant of the optimization in (\ref{opt-z_j}), studied in \cite{Javanmard2013b} is
\begin{align}\label{opt-z_j2}
\tbz_{j}= \matX\hbm, \qquad \hbm = \argmin_{\bm} \Big\{\bm^{T}\hbSigma\bm :\, \norm{\hbSigma\bm - \bfe_{j}}{\infty} \leq \lam''_{j}\Big\}. 
\end{align}
Since $\tbz_j^T\bx_j/n = 1 - \lam_j''$ and 
(\ref{opt-z_j}) is neutral in the sign of $\bz$, (\ref{opt-z_j2}) and (\ref{opt-z_j})
are equivalent with $\tbz_j/(1-\lam_j'') = \bz_j$ when $\lam_j = \lam_j''/(1-\lam_j'')$ and 
$\bz_j$ is the solution with $\bz_j^T\bx_j=n$. 

\subsection{Bias correction for a group of variables}\label{subsec:biascorrect-group} 
In this subsection we propose a multivariate extension of the methodologies 
described in Subsection \ref{subsec:biascorrect}. 

The algebraic extension of (\ref{LDPE-j}) to the grouped variable scenario is straightforward. 
For the estimation of $\bbetastar_G$, a formal vectorization of the estimator is 
\bel{LDPE-G}
\hbbeta_G = \hbbeta_G^{(init)} + (\matZ_G^T\matX_G)^\dag\matZ_G^T(\by -\matX\hbbeta^{(init)}), 
\eel
where $\matZ_G\in \bbR^{n\times|G|}$, depending on $\matX$ only, can be viewed as a ``score matrix''. 
Recall that for any matrix $\matA$, $\matA^{\dag}$ is its Moore-Penrose pseudo inverse. 
For the estimation of $\bmu_G^* = \matX_G\bbeta_G^*$, a variation of (\ref{LDPE-G}) is 
\bel{LDPE-G-pred}
\hbmu_G = \hbmu_G^{(init)} + (\matZ_G\matQ_G)^\dag\matZ_G^T(\by -\matX\hbbeta^{(init)}), 
\eel
where $\hbmu_G^{(init)} = \matX_G\hbbeta_G^{(init)}$ and 
$\matQ_G$ is the orthogonal projection to the column space $\matX_G$. 

The extension of the error decomposition (\ref{lasso:decomp}) to (\ref{LDPE-G}) and (\ref{LDPE-G-pred}) 
is also algebraic but requires a mild condition due to 
the need to factorize out a multivariate version of the noise factor. 
We carry out this task in the following proposition. 

\begin{proposition}\label{prop-1} Let $\matZ_G\in\bbR^{n\times|G|}$, 
$\matQ_A$ and $\matP_{G,0}$ be the orthogonal projections to $\calR(\matX_A)$ and 
$\calR(\matZ_G)$ respectively, $\matP_G$ be the orthogonal projection to $\calR(\matP_{G,0}\matQ_G)$, 
$\hbbeta_G$ be as in (\ref{LDPE-G}), $\hbmu_G=\matX_G\hbbeta_G$, 
$\bmu_A^* = \matX_A\bbeta_A^*$, 
$\hbmu_{A}^{(init)} = \matX_A\bbeta_A^{(init)}$, 
and  
\bel{prop-1-1}
\Rem_G 
= \sum_{\Gk\not\subseteq G}\matP_G\Big(\hbmu_{G_k\setminus G}^{(init)}-\bmu^*_{G_k\setminus G}\Big) 
= \sum_{\Gk\not\subseteq G}\Big(\matP_G\matQ_{G_k\setminus G}\Big)
\Big(\hbmu_{G_k\setminus G}^{(init)}-\bmu^*_{G_k\setminus G}\Big). 
\eel
(i) Suppose $\rank(\matZ_G^T\matX_G)=|G|$. Then, $\rank(\matP_G\matX_G)=|G|$, 
$\matP_G = \matP_{G,0}$, and 
\bel{prop-1-2} 
\hbbeta_G = \hbbeta_G^{(init)} + (\matP_G\matX_G)^\dag\matP_G\Big(\by -\matX\hbbeta^{(init)}\Big)
= \bbeta_G^* + (\matP_G\matX_G)^\dag 
\left(\matP_G\bepsa - \Rem_G\right). 
\eel
(ii) Suppose $\rank(\matP_G)=\rank(\matX_G)$. 
Then,  
(\ref{LDPE-G-pred}) holds and 
\bel{prop-1-3} 
\hbmu_G = \hbmu_G^{(init)} + (\matP_G\matQ_G)^\dag\matP_G\Big(\by -\matX\hbbeta^{(init)}\Big) 
= \bmu_G^* + (\matP_G\matQ_G)^\dag\left(\matP_G\bepsa - \Rem_G\right). 
\eel
Consequently, 
\bel{prop-1-4}
(\matP_G\matQ_G)(\hbmu_G - \bmu_G^*) =(\matP_G\matX_G)(\hbbeta_G-\bbeta_G^*) 
= \matP_G\bepsa - \Rem_G. 
\eel
In particular, when $\bmu_G^*={\bf 0}$, 
\bel{prop-1-5}
\matP_G\bepsa - \Rem_G = \matP_G\hbmu_G 
= \matP_G\bigg(\by - \sum_{\Gk\not\subseteq G}\hbmu_{G_k\setminus G}^{(init)}\bigg). 
\eel
\end{proposition}

The first equations of (\ref{prop-1-2}) and (\ref{prop-1-3}) assert the 
scale invariance of the proposed estimator in the choice of $\bZ_G$ in the sense that 
it depends in $\matZ_G$ only through the projection $\matP_G$. 

The condition $\rank(\matP_G)=\rank(\matX_G)$, 
slightly weaker than the condition $\rank(\matZ_G^T\matX_G)=|G|$, 
requires $\matZ_G^T\matX_G$ to have the same kernel as $\matX_G$. 
If this condition fails to hold, there will be no bias correction in 
a certain direction $\ba =\matX_G\bb_G \neq {\bf 0}$ in the sense that 
$\ba^T\hbmu_G = \ba^T\hbmu_G^{(init)}$. 


In Proposition \ref{prop-1}, the matrices $(\matP_G\matX_G)^\dag$ and $(\matP_G\matQ_G)^\dag$ and 
can be viewed as multivariate noise factors respectively 
for statistical inference of $\bbeta_G^*$ and $\bmu_G^*$, 
and the remainder term $\Rem_G$ can be viewed as standardized bias.  

For any estimator $\hsigma$ for the noise level and measurable function $h: \calR(\matP_G)\to\bbR$, 
\bel{pivotal-h}
h\big((\matP_G\matQ_G)(\hbmu_G - \bmu_G^*)/\hsigma\big) 
= h\big((\matP_G\matX_G)(\hbbeta_G-\bbeta_G^*)/\hsigma\big)
\eel
is an approximate pivotal quantity with approximate distribution $h(\matP_G\bepsa/\sigma)$ whenever 
\bel{asymp-normality}
\sup_{-\infty<t<\infty}
\Big|\bbP\Big\{h\big((\matP_G\bepsa- \Rem_G)/\hsigma\big)\le t\Big\} 
- \bbP\Big\{h\big(\matP_G\bepsa/\sigma\big)\le t\Big\}\Big| = o(1). 
\eel
From this point of view, the proposed method is generic. 
If a pivotal quantity (\ref{pivotal-h}) with a specific $h(\cdot)$ suits the aim of a statistical experiment, 
statistical inference can be carried out if certain estimator $\{\hbbeta^{(init)},\hsigma\}$ and score matrix 
$\matZ_G$ can be found to satisfy (\ref{asymp-normality}).  

As we are interested {}{in} chi-squared type inference, 
the right choice of $h(\cdot)$ is $h(\bv) = \|\bv\|_2$. 
This choice yields elliptical confidence regions for $\bbeta_G^*$ and $\bmu_G^*$ via (\ref{pivotal-h}). 
For testing the hypothesis $H_0:\bbeta_G=\bzero$, (\ref{prop-1-5}) provides the test statistic 
\bel{test}
T_G = \frac{1}{\hsigma}\left\|\matP_G\left(\by - 
\sum_{\Gk\not\subseteq G}\hbmu_{G_k\setminus G}^{(init)}\right)\right\|_2
\eel
as an approximation of $\|\matP_G\bepsa/\sigma\|_2$. 
Let $k_G= \rank(\matP_G)$. 
It is worthwhile to note that 
\bel{eq:largegrp}
\|\matP_G\bepsa\|_2/\sigma - \sqrt{k_G}\to \sfN(0,1/2)
\eel
when $k_G\to \infty$. 
Thus, without further investigation of possible stochastical cancellation between $\matP_G\bepsa$ 
and $\Rem_G$, (\ref{asymp-normality}) for $h(\bv) = \|\bv\|_2$ and $k_G\ge 1$ amounts to 
\bel{asymp-chi-sq}
\sqrt{k_G}\big|\hsigma/\sigma-1\big| + \big\|\Rem_G/\sigma\big\|_2 = o_{\bbP}(1). 
\eel
As $\|\matP_G\bepsa\|_2^2/\sigma^2$ has the $\chi_{k_G}^2$ distribution, (\ref{asymp-chi-sq}) implies 
\bel{eq:conv1}
\begin{cases}
\sup_t \Big|\bbP\Big\{ \|(\matP_G\matX_G)( \hbbeta_G- \bbeta_G^*)\|_2^2 \le \hsigma t\Big\}
 - \bbP\Big\{ \chi^2_{k_G} \le t\Big\}\Big| \to 0, 
 \cr 
\sup_t \Big|\bbP\Big\{ \|(\matP_G\matQ_G)( \hbmu_G- \bmu_G^*)\|_2^2 \le \hsigma t\Big\}
 - \bbP\Big\{ \chi^2_{k_G} \le t\Big\}\Big| \to 0, 
 \cr 
\bmu^*_G={\bf 0}\ \Rightarrow\ 
\sup_t \Big|\bbP\Big\{ T_G^2 \le t\Big\} - \bbP\Big\{ \chi^2_{k_G} \le t\Big\}\Big| \to 0. 
\end{cases}
\eel
When $k_G = \rank(\matP_G) \rightarrow \infty$, we can apply 
central limit theorem (\ref{eq:largegrp}) to approximate the $\chi^2_{k_G}$ distribution. 

The problem, as before, is to choose $\{\hbbeta^{(init)},\hsigma\}$ and $\matZ_G$ 
to guarantee (\ref{asymp-chi-sq}) for the given $h(\cdot)$. 
For definiteness, we will pick in the sequel the following scaled version of 
the group Lasso estimator (\ref{eq:grplsoopt-ns}): 
\bel{scaled-init}
\{\hbbeta^{(init)},\hsigma\}
= \argmin_{\bbeta,\sigma}\bigg\{\dfrac{\norm{\by-\matX\bbeta}{2}^{2}}{2n\sigma} 
+ \dfrac{\sigma}{2} + \sum^{M}_{j=1}\omega_{j}\norm{\bbeta_{G_j}}{2}\bigg\}.
\eel
This estimator, which aims to take advantage of the group sparsity (\ref{def:strnggrpsprse}), 
will be considered carefully in Section \ref{sec:consisres}, 
so that we can move on to the more pressing issue of finding a proper $\matZ_G$. 
Still, we would like to mention that this choice of $\{\hbbeta^{(init)},\hsigma\}$ 
and $h(\cdot)$ will in no way confine 
the scope of the proposed method, as Proposition \ref{prop-1} and (\ref{asymp-normality}) 
are completely general. 

\subsection{An ideal solution and a working assumption}\label{subsec:ideal-sol} 
To study the feasibility of the approach outlined above in Subsection \ref{subsec:biascorrect-group}, 
we first consider, parallel to (\ref{z_j^o}), an ideal $\matZ_G$ as the noise matrix in the following 
multivariate regression model, 
\bel{Gamma}
\matX_G = \matX_{-G}\bGamma_{-G,G}+\matZ_G^o. 
\eel
This regression model is best explained in the context of random design where 
\bel{Gamma-2}
\bGamma_{-G,G} =\big\{\bbE(\matX_{-G}^T\matX_{-G})\big\}^{-1}\bbE(\matX_{-G}^T\matX_{G}). 
\eel
To this end, we consider in the following theorem random design matrices $\matX$ 
having iid sub-Gaussian rows satisfying $\bbE\matX = {\bf 0}$, 
$\bbE(\matX^T\matX/n) = \bSigma$ with a positive-definite $\bSigma$, and 
\bel{subGaussian-cond}
(\textbf{\text{Sub-Gaussianity}})\qquad \sup_{\bb\neq {\bf 0}}\ \bbE \exp\left( \frac{(\bfe_i^T\matX\bb)^2}{v_0\bb^T\bSigma \bb} +\frac{1}{v_0} \right) \le 2
\eel
with a certain constant $v_0>1$, 
where $\bfe_{i} \in \Re^{n}$ is the $i^{th}$ canonical unit vector in $\bbR^n$. 

\begin{theorem}\label{th-ideal} 
Let $0<c_*\le c^*$ and $1 < A_*<A^*$ be fixed constants and 
$\{\hbbeta^{(init)},\hsigma\}$ be a solution of (\ref{scaled-init}) with 
$\omega_{j}/A^* \le \|\matX_{G_j}/\sqrt{n}\|_S\omega_{*,j} \le \omega_j/A_*$, 
where $\omega_{*,j} = n^{-1/2}(\sqrt{|G_j|} + \sqrt{2\log M})$. 
Suppose $\matX$ satisfies condition (\ref{subGaussian-cond}) with 
$c_*\le $eigenvalues$(\bSigma)\le c^*$. 
Let $\matZ_G^o$ be as in (\ref{Gamma}) with the $\bGamma_{-G,G}$ in (\ref{Gamma-2}) 
and $\hbbeta_G$ be as in (\ref{LDPE-G}) with $\matZ_G=\matZ_G^o$. 
Suppose $\by-\matX\bbeta^*\sim \sfN_{n}(\bzero,\sigma^{2}\matI_{n}\,)$ and $\bbeta^*$ satisfies 
the $(g,s)$ strong group sparsity condition (\ref{def:strnggrpsprse}) with 
\bel{th-ideal-1}
\frac{\max_{j\le M}|G_j|}{n}+\frac{|G|}{n} \to 0,\quad 
\dfrac{s+g \log M}{n^{1/2}}\left(\frac{|G|^{1/2}}{n^{1/2}}+ 
\max_{\Gk\not\subseteq G}\frac{\omega_k'}{\omega_{*,k}}\right) \rightarrow 0,\quad
\eel
where $\omega_k' = n^{-1/2}\big(\sqrt{|G|+|G_k\setminus G|}+\sqrt{\log M}\big)$. 
Then, $\bbP\{ \rank(\matP_G)=|G|\} \to 1$, (\ref{eq:conv1}) holds, and 
\bel{th-ideal-2}
(\matP_G\matQ_G)(\hbmu_G - \bmu_G^*)/\hsigma 
=(\matP_G\matX_G)(\hbbeta_G-\bbeta_G^*)/\hsigma 
= \sfN_{n}({\bf 0},\matP_G\,) + o_{\bbP}(1). 
\eel
\end{theorem}

Theorem \ref{th-ideal}, whose proof is merged with that of Theorem \ref{th:oraclerate} and provided 
in Subsection~\ref{subsec: Feasibility},
asserts that with a combination of the $\{\hbbeta^{(init)},\hsigma\}$ 
in (\ref{scaled-init}) and the ideal $\matZ_G=\matZ_G^o$ in (\ref{Gamma}), 
bias correction provides valid asymptotic chi-squared-type statistical inference for the group effect 
$\bmu_G^*\in\bbR^n$ and the coefficient group $\bbeta_G^*\in\bbR^{|G|}$. 
However, this theorem requires a sub-Gaussian design and the knowledge of $\matZ_G^o$. 

To extend this approach to more general settings with unknown $\matZ_G^o$ 
or even deterministic $\matX$, 
we follow a strategy parallel to the one described in Subsection \ref{subsec:biascorrect}: 
We may directly approximate $\matZ_G^o$ via a regularized multivariate regression in (\ref{Gamma}) 
or mimic properties of $\matZ_G^o$ with a regularized optimization scheme. The question is to make 
a right choice of the regularization on $\matZ_G$ to match properties one can reasonably expect 
from $\{\hbbeta^{(init)},\hsigma\}$. 
To this end, we extract, as the following working assumption, some properties 
{}{of} $\{\hbbeta^{(init)},\hsigma\}$ which are proven and used in our analysis 
under the conditions of Theorem \ref{th-ideal}.  

\smallskip
\noindent{\bf Working assumption:} 
{\it Suppose that we have estimators $\hbbeta^{(init)}$ and $\hsigma$ 
of a $(g,s)$ strong group sparse signal $\bbetastar$ and scale parameter $\sigma$ respectively 
satisfying 
\begin{align}\label{eq:prelimbetaconsis}
\left|\frac{\hsigma}{\sigma^*}-1\right| + 
\frac{1}{n^{1/2}}\sum_{j=1}^M \frac{\omega_{*,j}}{\sigma}
\Big\|\matX_{G_j}\hbbeta_{G_j}^{(init)}-\matX_{G_j}\bbeta^*_{G_j}\Big\|_2
= \calO_{\bbP}\left(\dfrac{s+g\log M}{n}\right),
\end{align}
where $\omega_{*,j} = \sqrt{|G_j|/n}+\sqrt{(2/n)\log M}$, 
$\sigstar = \norm{\matX\bbetastar-\by}{2}/\sqrt{n}$ is an oracle estimate of the noise level $\sigma$, 
and $G_j$, $s$ and $g$ are as in (\ref{def:strnggrpsprse}). }
\smallskip

The above working assumption still aims to {}{take} advantage of the group sparsity (\ref{def:strnggrpsprse}) 
as the mixed prediction error and the complexity measure $s+g\log M$ dictate. 
However, compared with the more specific (\ref{scaled-init}), it provides a direction for regularizing 
a proper $\matZ_G$ for any estimator satisfying (\ref{eq:prelimbetaconsis}), possibly with deterministic designs. 

Under the strong group sparsity (\ref{def:strnggrpsprse}), error bounds in the $\ell_{2}$ and 
mixed $\ell_{2,1}$ norms for group regularized methods have been established in the literature 
as we reviewed in the introduction. In Section \ref{sec:consisres}, we contribute to this literature by 
obtaining $\ell_{2}$ as well as weighted mixed $\ell_{2}$ norm error bounds of the group Lasso and 
its scaled version (\ref{scaled-init}). 
We will also provide a faster rate of convergence of the scale parameter $\sigma$ 
under strong group sparsity, which is crucial to our analysis. 
In particular, we will prove in Section \ref{sec:consisres} that the error bound for 
$\hbbeta^{(init)}$ in (\ref{eq:prelimbetaconsis}) is attainable under proper conditions on the 
design matrix if the group Lasso is used with a proper estimate of $\sigma$, 
and the error bounds for both $\hbbeta^{(init)}$ and $\hsigma$ in (\ref{eq:prelimbetaconsis}) are 
attainable if the scaled group Lasso is used; see Corollaries \ref{cor:grpsparse-ns} and \ref{cor-2} 
and Theorem~\ref{th-7}. 

It is worthwhile to point out that the working assumption exhibits the benefit of strong group sparsity, 
compared with a reasonable working assumption based on the $\ell_0$ sparsity condition 
$\|\bbetastar\|_0\le s$ as given in (\ref{lasso-error}). 
In general, the error bounds in (\ref{eq:prelimbetaconsis}) and those in (\ref{lasso-error}) do not 
strictly dominate each other. However, if in both the scenarios, $s$ is of similar order and $g\ll s$, 
then (\ref{eq:prelimbetaconsis}) dominates the rates necessary for univariate inference 
as given in (\ref{lasso-error}).

An alternative possibility is to use an $\ell_1$ regularized estimate of $\bGamma_{-G,j}$ in the univariate 
regression of $\bx_j$ against $\matX_{-G}$ for all individual $j\in G$. 
This has been considered in \cite{VDG14}. However, the advantage of such a scheme is unclear 
compared with directly using $(\hbeta_j,j\in G)^T$ with the $\hbeta_j$ in (\ref{LDPE-j}). 
It is worthwhile to mention that the central limit theorem for (\ref{LDPE-j}) came with large 
deviation bounds to justify Bonferroni adjustments \citep{Zhang2014}, so that 
(\ref{LDPE-j}) and its variations can be used to test $H_0:\bbeta^*_G=\bzero$ versus an  
alternative hypothesis on $\|\bbeta^*_G\|_\infty$, especially when 
an $\ell_1$ regularized $\hbbeta^{(init)}$ is used as in \cite{VandeGeer2013}. 
However, we are interested in extensions of traditional $F$- or chi-squared tests 
for $\ell_2$ alternatives and taking advantage of the group sparsity of $\bbeta^*$. 
Such methods require control of $\ell_{2}$ and groupwise weighted $\ell_{2}$ error and accordingly, 
a proper choice $\matZ_G$ to match the working assumption. 

\subsection{An optimization strategy} \label{subsec:optstrategy}
In this subsection we propose a multivariate extension of the optimization strategy 
(\ref{opt-z_j}) to match an initial estimator satisfying the working assumption (\ref{eq:prelimbetaconsis}) 
in the bias correction scheme (\ref{LDPE-G}). 

It follows from Proposition \ref{prop-1} that the estimator (\ref{LDPE-G}) depends on the resulting 
$\matZ_G$ only through the orthogonal projection $\matP_G$ to the range of $\matZ_G$ 
under a necessary assumption for the bias correction scheme to work, 
as we commented below Proposition~\ref{prop-1}. 
Moreover, it follows from (\ref{prop-1-1}) and (\ref{prop-1-4}) that the desired $\matP_G$, 
which depends on $\matX$ only, must be close to $\matQ_G$ and approximately orthogonal to 
$\matQ_{G_k\setminus G}$ for all $k$ with $\Gk\not\subseteq G$. 

Let $\matQ$ be the projection to $\calR(\matX)$. In the low-dimensional case of $\rank(\matX) = p<n$, 
we may set $\matP_G = \matQ\prod_{\Gk\not\subseteq G}\matQ_{G_k\setminus G}^\perp$, 
so that (\ref{LDPE-G}) is the least squares estimator of $\bbeta_G$ with $\Rem_G={\bf 0}$ 
in (\ref{prop-1-1}) and (\ref{prop-1-2}), and 
$T_G^2/|G|$ is the $F$-statistic for testing $H_0:\bbeta_G=\bzero$ when $\hsigma$ is 
the degree adjusted estimate of noise level based on the residuals of the least squares estimator. 
Of course, we need to relax the requirement of the orthogonality condition 
$\matP_G\matQ_{G_k\setminus G}=\bzero$ for all $\Gk\not\subseteq G$ in the high-dimensional case. 

Analytically, the key is to prove the upper bound $\|\Rem_G/\sigma\|_2 = o_{\bbP}(1)$ 
in (\ref{asymp-chi-sq}). 
To this end we use the formula in (\ref{prop-1-1}) and 
the working assumption in (\ref{eq:prelimbetaconsis}) to obtain  
\bel{eq:testC1} 
\|\Rem_G\|_2 
&\le& \left(\max_{\Gk\not\subseteq G}M_k\omega^{-1}_{*,k}\norm{\matP_G\matQ_{G_k}}{S}\right)
\sum_{\Gk\not\subseteq G}\omega_{*,k} \norm{\hbmu_{G_k}^{(init)}-\bmu_{G_k}}{2} 
\cr & = & \calO_{\bbP}\left(\dfrac{s+g\log M}{n^{1/2}}\right)
 \left(\max_{\Gk\not\subseteq G}M_k \omega^{-1}_{*,k}\norm{\matP_G\matQ_{G_k}}{S}\right), 
\eel
where $\omega_{*,k}=\sqrt{|G_k|/n}+\sqrt{(2/n)\log M}$ and 
$M_k=\max_{\|\matX_{G_k}\bu_{G_k}\|_2=1}\|\matX_{G_k\setminus G}\bu_{G_k\setminus G}\|_2$. 
We note that $M_k=1$ when $\matX_{G_k}^T\matX_{G_k}/n = \matI_{d_k\times d_k}$. 
Since $(s+g\log M)/n$ is the order of the mixed $\ell_{2,1}$ error bound for $\hbbeta$, 
we may treat $\eta_G = \max_{\Gk\not\subseteq G}M_k \omega^{-1}_{*,k}\norm{\matP_G\matQ_{G_k}}{S}$
as a scalar bias factor. 

The error bound in (\ref{eq:testC1}) motivates the following extension of (\ref{opt-z_j}): 
\begin{align}\label{opt}
\matP_G = \argmin_{\matP}\Big\{\|\matP\matQ_G^\perp\|_S: \matP = \matP^{2}= \matP^{T},\ 
\|\matP_G\matQ_{{\Gk\setminus G}}\|_S\le \omega_k'\ \forall\ \Gk\not\subseteq G\Big\}. 
\end{align}
We say that $\matP_G$ is a feasible solution of (\ref{opt}) if it satisfies all the constraints. The optimization problem (\ref{opt}) is a generalization of (\ref{opt-z_j}) and provides geometric insights. 
As $(\matP_G\matQ_{G})^\dag$ is a multivariate noise factor for the inference of $\bmu_G^*$, 
we may define $\tau_G = \|(\matP_G\matQ_{G})^\dag\|_S$ as a scalar noise factor. 
The quantity $\|\matP_G\matQ_{G}^\perp\|_S$, which 
is the so-called `gap' between the subspaces spanned by $\matP_G$ and $\matQ_{G}$, 
equals $(1-\tau_G^{-2})^{1/2}$. Thus, minimizing $\|\matP_G\matQ_{G}^\perp\|_S$ is equivalent to 
minimizing the noise factor $\tau_G$. 
This minimization is done subject to upper-bounds on the components 
$\norm{\matP_G\matQ_{{\Gk\setminus G}}}{S}$ of the bias factor. 
Thus, (\ref{opt}) is an extension of (\ref{opt-z_j}) as we discussed immediately after (\ref{opt-z_j}). 
When $p<n$ and $\omega_k'=0$, $\bP_G$ in (\ref{opt}) is the projection to 
the orthogonal complement of $\sum_{\Gk\not\subseteq G}\calR(\matX_{{\Gk\setminus G}})$ in 
$\calR(\matX)$, or equivalently the linear space 
$\big(\prod_{\Gk\not\subseteq G}\matQ_{{\Gk\setminus G}}^\perp\big) \calR(\matX)$. 

\smallskip
In the following theorem, we provide a summary of the analysis we have carried out above. 

\begin{theorem}\label{th-opt}
Let $\matP_G$ be a feasible solution of (\ref{opt}) satisfying $\|\matP_G\matQ_G^\perp\|_S<1$, 
and $\hbbeta_G$ be as in (\ref{LDPE-G}) with $\matZ_G = \matP_G$ 
and certain $\{\hbbeta^{(init)},\hsigma\}$ satisfying (\ref{eq:prelimbetaconsis}).  
Suppose $\bepsa\sim \sfN_{n}(\bzero,\sigma^{2}\matI_{n}\,)$, $\rank(\matX_G)=|G|$, and 
\bel{th-1-1}
\frac{|G|}{n} \to 0,\quad 
\dfrac{s+g \log M}{n^{1/2}}\left(\frac{|G|^{1/2}}{n^{1/2}}+ 
\max_{\Gk\not\subseteq G}M_k\frac{\omega_k'}{\omega_{*,k}}\right) \rightarrow 0,\quad
\eel
with the $M_k$ in (\ref{eq:testC1}). 
Then, (\ref{eq:conv1}) and (\ref{th-ideal-2}) hold. 
\end{theorem}

\begin{proof}[Proof of Theorem \ref{th-opt}] 
Since $\|\matP_G\matQ_G^\perp\|_S<1$, we have $\rank(\matP_G\matX_G)=\rank(\matX_G)=|G|$, 
so that the condition of Proposition \ref{prop-1} (i) holds, which implies the condition of 
Proposition \ref{prop-1} (ii) with $k_G=|G|$.   
It follows from (\ref{eq:prelimbetaconsis}), (\ref{eq:testC1}), (\ref{th-1-1}) 
and the feasibility of $\matP_G$ in (\ref{opt}) that (\ref{asymp-chi-sq}) holds, 
which implies (\ref{eq:conv1}) and (\ref{th-ideal-2}). 
Note that (\ref{eq:prelimbetaconsis}) and (\ref{th-1-1}) imply 
$\left|\sigma/\hsigma -1\right| = o_{\bbP}(|G|^{-1/2}) + O_{\bbP}(n^{-1/2})= o_{\bbP}(|G|^{-1/2})$ 
in the proof for the first component of (\ref{asymp-chi-sq}). 
\end{proof}

A modification of (\ref{opt}), which removes the factors $M_k$ in condition (\ref{th-1-1}), is to 
re-parameterize the effect of the $k$-th group by writing 
\bes
\matX_{G_k}\bbeta_{G_k} = {\widetilde {\matX}}_{G_k\cap G}\bbeta_{G_k\cap G}
+ \matX_{G_k\setminus G}{\widetilde \bbeta}_{G_k\setminus G}, 
\ees
where ${\widetilde {\matX}}_{G_k\cap G} = \matQ_{G_k\setminus G}^\perp\matX_{G_k\cap G}$ 
and ${\widetilde \bbeta}_{G_k\setminus G}$ is a solution of 
$\matX_{G_k\setminus G}{\widetilde \bbeta}_{G_k\setminus G} 
= \matQ_{G_k\setminus G}\matX_{G_k}\bbeta_{G_k}$. 
We recall that $\matQ_{G_k\setminus G}$ is the orthogonal projection to the column space of 
$\matX_{G_k\setminus G}$. As this within-group re-parameterization retains $\bbeta_{G_k\cap G}$ 
and $\matX_{G_k\setminus G}$, 
\bes
\by = {\widetilde {\matX}}_G\bbeta_G 
+ \sum_{\Gk\not\subseteq G} \matQ_{G_k\setminus G}\bmu_{G_k} + \bepsa
= {\widetilde {\matX}}_G\bbeta_G 
+ \sum_{\Gk\not\subseteq G}\matX_{G_k\setminus G}{\widetilde \bbeta}_{G_k\setminus G} + \bepsa,  
\ees
where ${\widetilde {\matX}}_G$ is the $n\times |G|$ matrix given by 
${\widetilde {\matX}}_G\bv_G = \sum_{k=1}^M \big(\matQ_{G_k\setminus G}^\perp\matX_{G_k\cap G}\big)
\bv_{G\cap G_k}$. 
As ${\widetilde {\matX}}_{G_k\cap G}$ is orthogonal to $\matX_{G_k\setminus G}$, 
we have $M_k=1$ after {}{re-parametrization}. 
Moreover, the strong group sparsity condition $\supp(\bbeta^*)\subset G_{S^*}$ 
and the working assumption (\ref{eq:prelimbetaconsis}) are invariant under the re-parameterization. 
We note that ${\widetilde {\matX}}_G = \matX_G$ when 
$\matX_{G_k}^T\matX_{G_k}/n = \matI_{G_k\times G_k}$ for all $k$ with $0<|G_k\setminus G|<|G_k|$.  
Let ${\widetilde {\matQ}}_G$ be the projection to the column space of ${\widetilde {\matX}}_G$. 
The optimization scheme and statistical methods are changed accordingly as follows: 
\bel{general-opt}
\matP_G &=& \argmin_{\matP}\Big\{\|\matP{\widetilde {\matQ}}_G^\perp\|_S: \matP = \matP^{2}= \matP^{T},\ 
\|\matP_G\matQ_{G_k\setminus G}\|_S\le \omega_k'\ \forall\ k\Big\}, 
\cr \hbbeta_G &=& (\matP_G{\widetilde {\matX}}_G)^\dag \matP_G
\left(\by - \sum_{\Gk\not\subseteq G} \matQ_{G_k\setminus G}\hbmu_{G_k}^{(init)}\right), 
\hbox{ when}\ \rank(\matP_G{\widetilde {\matX}}_G)=|G|,
\\ \nonumber T_G &=& \frac{1}{\hsigma}
\left\|\matP_G\left(\by - \sum_{\Gk\not\subseteq G} \matQ_{G_k\setminus G}\hbmu_{G_k}^{(init)}\right)\right\|_2. 
\eel
With $\{\matX_G,\matQ_G\}$ replaced 
by $\{{\widetilde {\matX}}_G,{\widetilde {\matQ}}_G\}$, 
our analysis yields the following theorem. 

\begin{theorem}\label{th-opt-gen} 
Let $\matP_G$, $\hbbeta_G$ and $T_G$ be given by (\ref{general-opt}) 
with $\|\matP_G{\widetilde {\matQ}}_G^\perp\|_S<1$.  
Suppose $\bepsa\sim \sfN_{n}(\bzero,\sigma^{2}\matI_{n}\,)$, $\rank(\matX_G)=|G|$, 
and (\ref{eq:prelimbetaconsis}) and (\ref{th-1-1}) hold with $M_k=1$.  
Then, (\ref{eq:conv1}) and (\ref{th-ideal-2}) hold with $\{\matX_G,\matQ_G\}$ replaced 
by $\{{\widetilde {\matX}}_G,{\widetilde {\matQ}}_G\}$. 
\end{theorem}

\begin{remark}\label{remark-1} 
It is worthwhile to note that Theorems \ref{th-opt} and \ref{th-opt-gen} only require a feasible solution satisfying $\|\matP_G{\matQ}_G^\perp\|_S<1$ and $\|\matP_G{\widetilde {\matQ}}_G^\perp\|_S<1$ respectively, which can be directly verified for any given $\matP_G$. Still, the optimality criterion on $\matP_G$ aims to have smaller confidence regions and more powerful tests through (\ref{eq:conv1}). 
In practice, it suffices to find a feasible solution with $\|\matP_G{\matQ}_G^\perp\|_S$ or 
$\|\matP_G{\widetilde {\matQ}}_G^\perp\|_S$ reasonably bounded away from 1. 
As the optimization problems in (\ref{opt}) and (\ref{general-opt}) 
are still somewhat abstract for the moment, 
in the following we prove the feasibility of $\matP_G$ in (\ref{opt}) for sub-Gaussian designs 
and describe penalized regression methods to find feasible solutions of (\ref{opt}) and (\ref{general-opt}).
\end{remark}

\subsection{Feasibility of relaxed orthogonal projection for random designs} \label{subsec: Feasibility} 
In this subsection, we discuss the existence of feasible solutions of the optimization in (\ref{opt}) 
for a sub-Gaussian design matrix satisfying (\ref{subGaussian-cond}) with 
$\bbE\matX = {\bf 0}$ and a positive-definite population Gram matrix $\bbE(\matX^T\matX/n) = \bSigma$. 
The feasibility is established under the assumption of the groupwise regression model as described in (\ref{Gamma}). 

We group the effects in the linear regression model (\ref{Gamma}) as follows: 
\bel{eq:mod1ForXj}
\matX_G = \matX_{-G}\bGamma_{-G,G}+\matZ_G^o  
= \sum_{k=1}^M \matX_{G_k\setminus G}\bGamma_{G_k\setminus G,G}+\matZ_G^o,
\eel
where $\bGamma_{-G,G} = \bSigma_{-G,-G}^{-1}\bSigma_{-G,G}$. 
Under this model assumption, $\matZg^{o}$ is the true residual after projection of $\matXg$ onto 
the range of $\matX_{-G}$. 
Let $\matP_G^o$ be the orthogonal projection to the column space of $\matZ_G^o$,  
\bel{P^o}
\matP_G^o = \matZ_G^o\Big((\matZ_G^o)^T\matZ_G^o\Big)^\dag(\matZ_G^o)^T. 
\eel
The following theorem establishes the distributional convergence results in (\ref{eq:conv1}) 
and (\ref{th-ideal-2}) for $\hbbetag$ by establishing the feasibility of $\matP_{G}^o$ 
as a solution of the optimization scheme in (\ref{opt}).

\begin{theorem}\label{th:oraclerate}
Suppose the sub-Gaussian condition (\ref{subGaussian-cond}) holds 
with $0<c_*\le$eigen$(\bSigma)\le c^*$ and fixed $\{v_0,c_*,c^*\}$. 
Let $\omega_k' = \xi n^{-1/2}\big(\sqrt{|G|+|G_k\setminus G|}+\sqrt{\log(M/\delta)}\big)$. \\
(i) Let $\lam_{\min}$ be the smallest eigenvalue of 
$\{\bSigma_{G,G}^{-1/2}(\bSigma^{-1})_{G,G}\bSigma_{G,G}^{-1/2}\}^{1/2}$,  and let 
$\xi n^{-1/2}\big(\sqrt{|G|}+\sqrt{\log(M/\delta)}\big)\le\eta_n$, 
and $a_n = \lam_{\min}(1-\eta_n)/(1+\eta_n)$. 
Then, there exist numerical constants $\eps_0\in (0,1)$ and $\xi_0<\infty$ such that when 
$\xi\ge \xi_0v_0$ and $\eta_n\le\eps_0$, 
\bel{th-3-1}
\bbP
\begin{cases}\hbox{\ (\ref{opt}) has a feasible solution $\bP_G$ with} \cr
\hbox{\ $\rank(\matP_G)=\rank(\matP_G\matX_G)=|G|$ and $\|\bP_G\matQ_G^\perp\|_S\le \sqrt{1-a_n^2}$}\end{cases}\Bigg\}
\ge 1-\delta. 
\eel
(ii)  Suppose the strong sparsity condition the sample size condition (\ref{th-ideal-1}) hold 
and that $\{\hbbeta^{(init)},\sigma\}$ is as in Theorem~\ref{th-ideal}. Then, 
the working assumption (\ref{eq:prelimbetaconsis}) holds. \\
(iii) Suppose the working assumption (\ref{eq:prelimbetaconsis}) 
and the sample size condition (\ref{th-ideal-1}) hold. 
Then, (\ref{eq:conv1}) and (\ref{th-ideal-2}) hold. 
\end{theorem}

Theorem \ref{th:oraclerate} removes the requirement of the knowledge of $\matZ_G^o$ 
in Theorem \ref{th-ideal}. 
It shows the existence of at least one feasible solution of (\ref{opt}) 
and that for such a choice of $\matP_G$, the $\chi^{2}$ based inference can be carried out as in 
(\ref{eq:conv1}) and (\ref{th-ideal-2}). 
However, (\ref{opt}) is not a convex program.  
In Subsection \ref{subsec:findSolution} we will describe group Lasso programs as convexation of (\ref{opt}). 

The proof of Theorem \ref{th:oraclerate} requires the following lemma on the probabilistic control of 
the spectral norm of the product of two random matrices with sub-Gaussian rows. 
As an extension of that result, spectral norm control of the product of two orthogonal projection 
matrices is also obtained. These probabilistic bounds in Lemma \ref{lm-subGaussian} are of independent interest. See Remark \ref{rem:lemma1} for more details.

\begin{lemma}\label{lm-subGaussian} Let $\matB_k$ be deterministic matrices with with $p$ rows and $\rank(\matB_k)= r_k$ for $k=\{1,2\}$. 
Let $\matP_k$ be the projection to the range of $\matX\matB_k$ and 
\[\bOmega_{1,2} = ((\matB_1^{T}\bSigma\matB_1)^\dag)^{1/2}\matB_1^{T}\bSigma\matB_2
((\matB_2^{T}\bSigma\matB_2)^\dag)^{1/2}.\] 
Let $r = \rank(\bOmega_{1,2})$ and $1\ge\lam_1\ge\cdots\ge \lam_r>0$ be the nonzero 
singular values of $\bOmega_{1,2}$. Define $\lam_{\min}=\lam_r I\{r=r_1=r_2\}$. 
Then, there exists a numerical constant $C_0>1$ such that 
when $C_0v_0\sqrt{t/n+(r_1+r_2)/n}<\eps_0<1$, 
\bel{lm-1-1}
\bbP\left\{\|((\matB_1^{T}\bSigma\matB_1)^\dag)^{1/2}\matB_1^{T}(\matX^T\matX/n)\matB_2
((\matB_2^{T}\bSigma\matB_2)^\dag)^{1/2}-\bOmega_{1,2}\|_S \le \eps_0\right\} \ge 1- e^{-t},  
\eel
and
\bel{lm-1-2}
\bbP\left\{\|\matP_1\matP_2\|_S \le \frac{\lam_1(1+\eps_0)}{1-\eps_0}, 
\|\matP_1\matP_2^\perp\|_S^2 \le 1 - \left(\frac{\lam_{\min}(1-\eps_0)}{1+\eps_0}\right)^2 \right\} \ge 1- e^{-t}. 
\eel
Moreover, $\lam_1<1$ iff $\rank(\matB_1,\matB_2)=r_1+r_2$ 
and $\lam_{\min}>0$ iff $\rank(\matB_1^T\matB_2)=r_1=r_2$. 
\end{lemma}

We have moved the proof of Lemma \ref{lm-subGaussian} to the \nameref{sec:appendix} 
to avoid a distraction from the main results of this section. 
Based on Lemma \ref{lm-subGaussian}, we prove Theorems \ref{th-ideal} and \ref{th:oraclerate} as follows. 

\begin{proof}[Proofs of Theorems \ref{th-ideal} and \ref{th:oraclerate}]
By (\ref{P^o}), $\matP_G^o$ is the orthogonal projection to the range of $\matZ_G^o = \matX\matB_G^o$ 
with $\matB_G^o=(\bSigma^{-1})_{*,G}(\bSigma^{-1})_{G,G}^{-1}$. 
By definition, $\matQ_{G_k\setminus G}$ is the projection to the range of 
$\matX_{G_k\setminus G} = \matX\matB_{G_k\setminus G}$ 
and $\matQ_{G}$ to the range of $\matX_{G}=\matX\matB_G$, 
where $\matB_{G_k\setminus G}$ and $\matB_G$ are 0-1 diagonal matrices projecting to the indicated spaces. 
Define $\bOmega = \bSigma_{G,G}^{-1/2}\big\{(\bSigma^{-1})_{G,G}\}^{1/2}$.  
We have 
$\matB_{G_k\setminus G}^T\bSigma\matB_G^o=\bSigma_{G_k\setminus G,*}\matB_G^o=0$,  
$\matB_G^T\bSigma\matB_G^o=\bSigma_{G,*}\matB_G^o = (\bSigma^{-1})_{G,G} 
= (\matB_G^o)^T\bSigma\matB_G^o$ and 
\bes
(\matB_G^T\bSigma\matB_G)^{-1/2}\matB_G^T\bSigma\matB_G^o
\big((\matB_G^o)^T\bSigma\matB_G^o\big)^{-1/2} 
= \bSigma_{G,G}^{-1/2}\big\{(\bSigma^{-1})_{G,G}\}^{1/2} = \bOmega \in\bbR^{|G|\times |G|}. 
\ees
Moreover, $\bOmega = \bSigma_{G,G}^{-1/2}\big\{(\bSigma^{-1})_{G,G}\}^{1/2}$ is a $|G|\times |G|$ 
matrix of rank $|G|$ and the smallest 
singular value of $\bOmega$ is $\lam_{\min}$. 
Thus, by (\ref{lm-1-2}) of Lemma \ref{lm-subGaussian} and the definition of $\omega_k'$ and $a_n$, 
\bes
\bbP\Big\{\|\bP_G\matQ_{G_k\setminus G}\|_S\le\omega_k'\ \forall k\le M,\ 
\|\bP_G\matQ_G^\perp\|_S\le \sqrt{1-a_n^2}\Big\}
\ge 1-\delta. 
\ees
This yields (\ref{th-3-1}). Moreover, (\ref{th-3-1}) also holds when $\matP_G = \matP_G^o$ 
or equivalently $\matZ_G=\matZ_G^o$ is used as in Theorem \ref{th-ideal}. 
As part (ii) of Theorem \ref{th:oraclerate} restates Theorem \ref{th-7} in Section \ref{sec:consisres}, 
it remains to prove $\max_{{G_k\setminus G \neq \emptyset}}M_k=O_{\bbP}(1)$ 
in view of Theorem \ref{th-opt}. To this end, we notice that due to the condition 
$|G_k| + g\log M \ll n$, (\ref{lm-1-1}) of Lemma \ref{lm-subGaussian} with $\matB_1=\matB_2$
implies $\|\matX_{A}^T\matX_{A}/n - \bSigma_{A,A}\|_S=o_{\bbP}(1)$ for both 
$A=G_k$ and $A=G_k\setminus G$ and all $k$ with $G_k\setminus G \neq \emptyset$,
so that $\max_{{G_k\setminus G \neq \emptyset}}M_k = o_{\bbP}(1)+O(1)$. 
\end{proof}

\begin{remark}\label{rem:lemma1}
Since Lemma \ref{lm-subGaussian} is a crucial ingredient for Theorems \ref{th-ideal} and \ref{th:oraclerate}, 
we highlight a few key points. Let us write $p=p_{1}+p_{2}$ and $\matI_{p}= [\matI_{p\times p_{1}}\, \matI_{p\times p_{2}}]$. Consider the choices: $\matB_{1}= \matI_{p\times p_{1}}$ and $\matB_{2}= \matI_{p\times p_{2}}$. Also consider the partition $\matX= [\matX_{1}\, \matX_{2}]$ so that $\matX_{i}= \matX\matB_{i}$. Writing 
\begin{align*}
\bSigma = \begin{bmatrix}
\bSigma_{11} & \bSigma_{12}\\
\bSigma_{21} & \bSigma_{22}
\end{bmatrix}\text{ where } \bSigma_{11} \in \Re^{p_{1}\times p_{1}}, \bSigma_{12} \in \Re^{p_{1}\times p_{2}}, \bSigma_{22} \in \Re^{p_{2}\times p_{2}},
\end{align*}
it follows that ${\rm cov} (\matX_{1}, \matX_{2})= \bSigma_{12}$. For such choices, Lemma \ref{lm-subGaussian} gives,
\begin{align}\label{eq:crossprodspecnorm}
\norm{\bSigma^{-1/2}_{11}\Big({\matX_{1}^{T}\matX_{2}}/{n}- \bSigma_{12}\Big)\bSigma_{22}^{-1/2}}{S} \leq C \sqrt{{t}/{n} + {(p_{1}+p_{2})}/{n}}
\end{align}
with probability at least $1-e^{-t}$. This result provides a spectral norm bound on the cross-product of two correlated random matrices with sub-Gaussian rows. The probability bound in (\ref{eq:crossprodspecnorm}) is a generalization of a similar result for product of two mutually independent random matrices with iid $\sfN(0,1)$ entries, given in Proposition D.1 in the supplement to \cite{Ma13}. Control of spectral norm of product of random and deterministic matrices have been studied as well; see \cite{Ver08}, \cite{RudVer13} etc. In particular, spectral norm concentration of product of a fixed projection matrix and a random matrix have been derived in \cite[Remark~3.3]{RudVer13}. In comparison, our results in (\ref{lm-1-2}) studies product of two projection matrices with their range being column spaces of correlated random matrices with sub-Gaussian rows.
\end{remark}

\subsection{Finding feasible solutions and construction of tests} \label{subsec:findSolution}
While (\ref{th-3-1}) of Theorem~\ref{th:oraclerate} guarantees a feasible solution of (\ref{opt}), the practicality of the optimization scheme (\ref{opt}) has not yet been addressed.  
We discuss here penalized multivariate regression methods for finding feasible solutions of 
(\ref{opt}) and (\ref{general-opt}). 
As the only difference between (\ref{opt}) and (\ref{general-opt}) is the respective use of 
$\matX_G$ and ${\widetilde \matX}_G$, we provide formulas here only for (\ref{opt}), with the understanding that 
formulas for (\ref{general-opt}) can be generated in the same way with 
$\matX_G$ replaced by ${\widetilde \matX}_G$. 

The optimization problem in (\ref{opt}) is carried out over the non-convex space of orthogonal projection matrices. In the following, we provide a convex program for obtaining such orthogonal projection matrices under the linear regression framework of (\ref{eq:mod1ForXj}). 
In model (\ref{eq:mod1ForXj}), a general formulation of the penalized multivariate regression is 
\bel{pen-multi-reg}
\hbGamma_{-G,G} = \argmin_{\bGamma_{-G,G}}
\left\{\frac{1}{2n}\left\|\matX_G - \sum_{G_k\not\subseteq G} 
\matX_{G_k\setminus G}\bGamma_{G_k\setminus G,G}\right\|_F^2
+ R(\bGamma_{-G,G})\right\}, 
\eel
where $\|\cdot\|_F$ is the Frobenius norm and $R(\bGamma_{-G,G})$ is a penalty function. Define 
\bel{pen-multi-Z}
\matZ_G = \matX_G - \sum_{G_k\not\subseteq G}  
\matX_{G_k\setminus G}\hbGamma_{G_k\setminus G,G},\quad 
\matP_G = \matZ_G(\matZ_G ^T  \matZ_G)^{-1}\matZ_G ^T. 
\eel
Our main interest is to find a feasible solution of (\ref{opt}) and (\ref{general-opt}), 
not to estimate $\bGamma_{-G,G}$. The following weighted group nuclear penalty matches the dual of the constraint 
in (\ref{opt}) and (\ref{general-opt}): 
\bel{n-pen}
R(\bGamma_{-G,G}) = \sum_{G_k\not\subseteq G} \frac{\xi \omega_k''}{n^{1/2}}
\Big\|\matX_{G_k\setminus G}\bGamma_{G_k\setminus G,G}\Big\|_N. 
\eel
Recall that nuclear norm of a matrix $\matA$, denoted $\norm{\matA}{N}$, is the sum of absolute values of the singular values of $\matA$. It follows from the KKT conditions for (\ref{pen-multi-reg}) with (\ref{n-pen}) that 
\bel{dual}
\left\|\matQ_{G_k\setminus G}\matZ_G/\sqrt{n}\right\|_S \le \xi\omega_k''. 
\eel
If we set $\omega_k''=\omega_k$ in (\ref{n-pen}), condition (\ref{th-1-1}) follows from  
\bel{th-2-1a}
\frac{|G|}{n} \to 0,\quad 
\dfrac{s+g \log M}{n^{1/2}}\left(\frac{|G|^{1/2}}{n^{1/2}} +  
\xi \|(\matZ_G^T\matZ_G/n)^{-1/2}\|_S\right) \to 0, 
\eel
provided $\max_{G_k\not\subseteq G}M_k=O(1)$ in the case of Theorem \ref{th-opt}.  Moreover, as in \cite{VDG14}, under the assumption $\lambda_{\min}(\matZ_{G})>c>0$, only 
${(s+g \log M)}/{n^{1/2}} + |G|/n \to 0$ suffices.

When the group sizes are not too large, one may consider replacing the weighted group nuclear penalty with a weighted group Frobenius penalty:
\bel{eq:frobpen}
R(\bGamma_{-G,G}) = \sum_{G_k\not\subseteq G} 
\frac{\xi \omega_k''}{n^{1/2}}\Big\|\matX_{G_k\setminus G}\bGamma_{G_k\setminus G,G}\Big\|_F.
\eel
The KKT conditions for (\ref{pen-multi-reg}) with (\ref{eq:frobpen}) yield 
\bes
\left\|\matQ_{G_k\setminus G}\matZ_G/\sqrt{n}\right\|_S 
\le \left\|\matQ_{G_k\setminus G}\matZ_G/\sqrt{n}\right\|_F \le \xi\omega_k'', 
\ees
so that (\ref{th-2-1a}) is still valid. 
However, this second layer of inequality indicates that the resulting procedure may not be as efficient 
as the (\ref{n-pen}) penalty.  
In any case, as discussed in Remark~\ref{remark-1}, 
it is reasonable to proceed with the computed $\bZ_G$ 
as long as the resulting $\|\matP_G\matQ_G^\perp\|_S$ is not too close to 1.   
One important benefit of the formulation of the groupwise penalty as in (\ref{eq:frobpen}) is that it can be conveniently computed using the standard group Lasso algorithms; see \cite{Yuan2006}, \cite{HuangBM12} etc. As we will show in Section \ref{sec:simures}, group Lasso performs well for empirical studies. 
We summarize {}{our} proposal and main results as follows. 

\smallskip
\noindent \textbf{Summary: } Statistical inference for groups of variables can be carried out as follows:
\begin{itemize}
\item Given $(\by,\matX)$ and a group structure $\{G_{j}:1\leq j\leq M\}$, construct the initial estimates 
$(\hbbeta^{(init)}, \hsigma)$ via the scaled group Lasso (\ref{scaled-init}) or any alternative leading to (\ref{eq:prelimbetaconsis}).
\item Given a variable group $G$ of interest, construct relaxed projection estimate 
$\matP_{G}= \matZg(\matZg^{T}\matZ_{g})^{-1}\matZ_{G}^{T}$ 
by the penalized procedure (\ref{pen-multi-reg}) and (\ref{pen-multi-Z}) 
with the penalty function (\ref{n-pen}) or (\ref{eq:frobpen}). 
\item Carry out statistical inference according to (\ref{eq:conv1}) and (\ref{th-ideal-2})
\end{itemize}

\smallskip
\noindent \textbf{Benefit of group sparsity: } Existing sample size condition for statistical inference 
of a univariate parameter at $n^{-1/2}$ rate requires, 
\[n \gg \|\bbetastar\|_0^2(\log p)^2.\]
See for exampe \cite{Zhang2014, VandeGeer2013, Javanmard2013b}. 
As discussed below (\ref{expansion}), direct application of these results to approximate chi-square 
group inference requires an extra factor $|G|$: 
\[n \gg |G|\times \|\bbetastar\|_0^2(\log p)^2.\]
If the true parameter $\bbetastar$ is $(g,s)$ strong group sparse with $s \asymp \norm{\bbetastar}{0}$, the sample size conditions in (\ref{th-1-1}), (\ref{th-ideal-1}) and (\ref{th-2-1a}) clearly demonstrate the benefit of group sparsity by incorporating the smaller estimation error bound as in \cite{HZ10} and removing the extra 
$|G|$. In particular, our sample size requirement becomes the much weaker 
\[n \gg \big(s  + g\log p\big)^2\]
for approximate chi-square inference when 
$|G|\lesssim \min_{G_k\not\subseteq G}\{|G_k|+\log(M/\delta)\}$ in (\ref{th-ideal-1}) or 
$\xi \|(\matZ_G^T\matZ_G/n)^{-1/2}\|_S=O(1)$ in (\ref{th-2-1a}). 


\section{Verification of Working Assumption}\label{sec:consisres} The analysis in the preceding section established the benefits of grouping in constructing $\ell_{2}$ type statistical inference procedures 
for variable groups. One key aspect of our analysis was the working assumption in (\ref{eq:prelimbetaconsis}). These results showed a faster convergence rate for the scale parameter estimate and the coefficient parameter estimate. As promised, in this section we will establish the bona fides of (\ref{eq:prelimbetaconsis}) under the strong group sparsity assumption in (\ref{def:strnggrpsprse}).

Generally, for high dimensional regression problems, certain regularity conditions on the the design matrix is required for estimation as well as prediction consistency. In the following Subsection \ref{subsec:designassump}, we discuss similar assumptions on the design matrix $\matX$ that ensure the consistency results in (\ref{eq:prelimbetaconsis}). We also derive estimation and prediction consistency result for the non-scaled group Lasso problem in (\ref{eq:grplsoopt-ns}) in Theorem \ref{th:betal12-ns} as an illustration. The main result of this section is Theorem \ref{th:betasgrp} and Corollary \ref{cor:grpsparse-ns} in Subsection \ref{subsec:sgrplasso} and Theorem \ref{th-7} in Subsection~\ref{sub-3-3} that establish the working assumption (\ref{eq:prelimbetaconsis}).

\subsection{Group Lasso and conditions on the design matrix}\label{subsec:designassump} 
In the Lasso problem, 
performance bounds of the estimator are derived based on various conditions on the design matrix, 
for example, the restricted isometry property \citep{CandesT05}, 
the sparse Riesz condition \citep{ZhangH08}, 
the restricted eigenvalue condition \citep{BRT09, Koltchinskii09}, 
the compatibility condition \citep{VDG07, VDG09}, 
and cone invertibility conditions \citep{YeZ10}. 
\cite{VDG09} showed that the compatibility condition is weaker than the restricted eigenvalues condition 
for the prediction and $\ell_1$ loss, 
while \cite{YeZ10} showed that both conditions can be weakened by cone invertibility conditions. 
In the following, we define grouped versions of such conditions, 
which will be used in our study.  

Let us first define a groupwise mixed norm cone for $T\subset \{1\cdots, M\}$ and $\xi\ge 0$ as 
\begin{align}\label{eq:cone1}
\scrC^{(G)}(\xi,\bomega,T) = \Big\{\bu: \hbox{$\sum_{j\in T^{c}}$}\omega_{j}\norm{\bugj}{2} 
\leq \xi \hbox{$\sum_{j\in T}$} \omega_{j}\norm{\bugj}{2}\neq 0\Big\}.
\end{align}
Let $T^* =\{1\cdots, M\}$ and $T\subseteq T'\subseteq T^*$.  
Following \cite{NR08} and \cite{Lounici2011}, the restricted eigenvalue (RE) is defined as 
\begin{align}\label{eq:RE}
{\rm RE}^{(G)}(\xi, \bomega, T,T') = \inf_{\bu} \bigg\{\dfrac{\norm{\matX\bu}{2}}{\sqrt{n}\norm{\bu_{G_{T'}}}{2}}:
\bu \in \scrC^{(G)}(\xi,\bomega, T)\bigg\}.
\end{align}
For the weighted $\ell_{2,1}$ norm, 
the groupwise compatibility constant (CC) can be defined as 
\begin{align}\label{eq:CC}
{\rm CC}^{(G)}(\xi, \bomega, T) 
= \inf_{\bu} \bigg\{\dfrac{\norm{\matX\bu}{2}\big(\sum_{j\in T}\omega^{2}_{j}\big)^{1/2}}
{\sqrt{n}\sum_{j\in T}{\omega_{j}\norm{\bu_{\Gj}}{2}}}: 
\bu \in \scrC^{(G)}(\xi,\bomega, T)\bigg\}.
\end{align}
We note that ${\rm RE}^{(G)}(\xi, \bomega, T,T)$ and the somewhat larger 
${\rm CC}^{(G)}(\xi, \bomega, T)$ are aimed at the prediction and the weighed $\ell_{2,1}$ estimation errors, 
while the smaller ${\rm RE}^{(G)}(\xi, \bomega, T,T^*)$ is aimed at the $\ell_2$ estimation error. 

We also introduce the notion of groupwise cone invertibility factor 
and its sign-restricted version. For $q\ge 1$, the cone invertibility factor (CIF) is defined as
\begin{align}\label{eq:CIF}
{\rm CIF}^{(G)}_q(\xi,\bomega,T,T') 
= \inf_{\bu \in \scrC^{(G)}(\xi,\bomega,T)}
\dfrac{\max_{j}\left[\omega^{-1}_{j}\norm{\matXgj^{T}\matX\bu}{2}\right] 
\big(\sum_{j\in T}\omega^{2}_{j}\big)^{1/q}}
{n \big(\sum_{j\in T'}{\omega_{j}^2(\norm{\bu_{\Gj}}{2}/\omega_j)^q}\big)^{1/q}}. 
\end{align}
We note that $\big(\sum_{j\in T'}{\omega_{j}^2(\norm{\bu_{\Gj}}{2}/\omega_j)^q}\big)^{1/q}=\|\bu\|_2$ 
when $T'=T^*$ and $q=2$.  Define 
\begin{align}\label{eq:cone2}
\scrC^{(G)}_{-}(\xi,\bomega, T) = \left\{\bu: \bu \in \scrC^{(G)}(\xi,\bomega, T),\ 
\bugj^{T}\matXgj^{T}\matX\bu \leq 0\ \forall j \in T^{c}\right\}, 
\end{align}
as a sign-restricted cone. We extend the CIF to 
the groupwise sign-restricted cone invertibility factor (SCIF) as
\begin{align}\label{eq:SCIF}
\SCIF^{(G)}_q(\xi,\bomega,T,T') = \inf_{\bu \in \scrC^{(G)}_{-}(\xi,\bomega,T)}
\dfrac{\max_{j}\left[\omega^{-1}_{j}\norm{\matXgj^{T}\matX\bu}{2}\right] 
\big(\sum_{j\in T}\omega^{2}_{j}\big)^{1/q}}
{n \big(\sum_{j\in T'}{\omega_{j}^2(\norm{\bu_{\Gj}}{2}/\omega_j)^q}\big)^{1/q}}. 
\end{align}

Similar to the RE and CC, 
${\rm CIF}^{(G)}_1(\xi,\bomega,T,T)$ and $\SCIF^{(G)}_1(\xi,\bomega,T,T)$ are 
aimed at the prediction and weighted $\ell_{2,1}$ losses, 
while ${\rm CIF}^{(G)}_q(\xi,\bomega,T,T^*)$ and  
$\SCIF^{(G)}_q(\xi,\bomega,T,T^*)$ is aimed at the weighted $\ell_{2,q}$ loss 
$\big(\sum_{j=1}^M{\omega_{j}^2(\norm{\bu_{\Gj}}{2}/\omega_j)^q}\big)^{1/q}$. 
We note that the weighted $\ell_{2,q}$ norm is identical to the $\ell_2$ norm for $q=2$. 
For $\bu\in \scrC^{(G)}_{-}(\xi,\bomega,T)$, 
\bes
\norm{\matX\bu}{2}^2\big/
\max_j(\omega_j^{-1}\norm{\matXgj^{T}\matX\bu}{2}) \le \hbox{$\sum_{j\in T}$} \omega_j\norm{\bu_{\Gj}}{2}
\le \|\bu_{G_T}\|_2\big(\hbox{$\sum_{j\in T}$} \omega_j^2\big)^{1/2}
\ees 
by the sign restriction and the Cauchy-Schwarz inequality, so that 
\bel{eq:constineq}
& \{{\rm RE}^{(G)}(\xi, \bomega, T,T)\}^{2} \leq \{{\rm CC}^{(G)}(\xi, \bomega, T)\}^{2} 
\le \SCIF^{(G)}_1(\xi,\bomega,T,T), 
\cr & {\rm RE}^{(G)}(\xi, \bomega, T,T^*){\rm CC}^{(G)}(\xi, \bomega, T)
\le \SCIF^{(G)}_2(\xi,\bomega,T,T^*). 
\eel
For $\bu\in \scrC^{(G)}(\xi,\bomega,T)$, 
$\SCIF^{(G)}_q(\xi,\bomega,T,T')$ can be replaced by $(\xi+1)\ {\rm CIF}^{(G)}_q(\xi,\bomega,T,T')$
in (\ref{eq:constineq}), as $\norm{\matX\bu}{2}^2\big/
\max_j(\omega_j^{-1}\norm{\matXgj^{T}\matX\bu}{2}) \le 
\sum_j \omega_j\norm{\bu_{\Gj}}{2} \le (1+\xi)\sum_{j\in T} \omega_j\norm{\bu_{\Gj}}{2}$. 
Thus, if a restricted eigenvalue condition holds in the sense of 
$\{{\rm RE}^{(G)}(\xi, \bomega, T)\}^2>\kappa_0$ with a fixed $\kappa_0$, then 
all the other quantities in (\ref{eq:constineq}) and $(\xi+1)\ {\rm CIF}^{(G)}_q(\xi,\bomega,T)$ 
are bounded from below by $\kappa_0$, $q\in \{1,2\}$. 
It follows that the cone invertibility factors provide error bounds of sharper form than (\ref{eq:RE}), 
in view of Theorem~\ref{th:betal12-ns} below and Theorem 3.1 of \cite{Lounici2011}. 

In the following Theorem \ref{th:betal12-ns} we provide the prediction, $\ell_2$ and 
mixed norm consistency results for the non-scaled group Lasso problem 
defined in (\ref{eq:grplsoopt-ns}) under the SCIF condition. 

\begin{theorem}\label{th:betal12-ns}
Let $\hbbeta=\hbbeta(\bomega)$ be a solution of (\ref{eq:grplsoopt-ns}) 
with data $(\matX,\by)$ and 
$\bbeta^*$ be a vector with $\supp(\bbetastar) \subseteq G_{\Sstar}$ for some 
$\Sstar \subset T^*=\{1,\cdots, M\}$. 
Let $\xi>1$ and define 
\begin{align}
\calE = \left\{\max_{1\le j\le M}\frac{\norm{\matXgj^{T}(\by-\matX\bbetastar)}{2}}
{\omega_{j}n} \leq \dfrac{\xi-1}{\xi+1}\ \right\}.
\end{align}
Then in the event $\calE$, we have
\bel{gl-pred}
\norm{\matX\hbbeta-\matX\bbeta^{*}}{2}^2/n 
\le \frac{\{2\xi/(\xi+1)\}^2\sum_{j \in \Sstar}\omega^{2}_{j}}{\SCIF^{(G)}_1(\xi, \bomega, \Sstar,\Sstar)},  
\eel
and for all $q\ge 1$
\begin{align}\label{eq:betal12-ns}
\bigg\{\sum^{M}_{j=1}\omega_{j}^2\bigg(\frac{\norm{\hbbetagj-\bbeta_{\Gj}^{*}}{2}}
{\omega_{j}}\bigg)^q\bigg\}^{1/q}
\leq \dfrac{\{2\xi/(\xi+1)\}\big(\sum_{j \in \Sstar}\omega^{2}_{j}\big)^{1/q}}
{\SCIF^{(G)}_q(\xi, \bomega, \Sstar,T^*)}. 
\end{align}
Moreover, if $\by-\matX\beta^*\sim \sfN_{n}(\bzero,\sigma^{2}\matI_{n}\,)$ and 
$\omega_{j} \ge A\sigma\|\matX_{G_j}\|_{S}
\big\{|G_j|^{1/2} + \sqrt{2\log (M/\delta)}\big\}/n$ 
for some $0<\delta <1$ and $A\geq (\xi+1)/(\xi-1)$, then
\begin{align}\label{eq:Eprob-ns}
\bbP(\calE) > 1-\delta. 
\end{align}
\end{theorem}

Theorem \ref{th:betal12-ns} asserts that the prediction loss 
$\norm{\matX\hbbeta-\matX\bbeta^{*}}{2}^2/n$, the $\ell_2$ loss 
$\|\hbbeta-\bbeta^*\|_2^2$ and the mixed norm loss 
$\sum^{M}_{j=1}\omega_{j}\norm{\hbbetagj-\bbeta_{\Gj}^{*}}{2}$ 
are all of the order 
\bes
\hbox{$\sum_{j\in S^*}$}\omega_j^2 \asymp (s+g\log M)/n
\ees 
when the SCIF can be treated as constant and $\max_j\|\matX_{G_j}/\sqrt{n}\|_S=O_{\bbP}(1)$.  
This result illustrates the benefit of the group Lasso as compared to Lasso. 
The results in Theorem \ref{th:betal12-ns} are not entirely new. 
In fact, for the group Lasso problem (\ref{eq:grplsoopt-ns}), the same convergence rate 
can be derived from the $\ell_{2}$ consistency result in \cite{HZ10}. 
While the result of \cite{HZ10} is derived under a sparse eigenvalue condition 
on the design matrix $\matX$, our results are based on the weaker sign-restricted cone invertibility 
condition and cover the weighted $\ell_{2,q}$ loss for $q>2$.  
The proof of Theorem \ref{th:betal12-ns} is relegated to the \nameref{sec:appendix}. 

\subsection{A scaled group Lasso}\label{subsec:sgrplasso}
In the optimization problem (\ref{eq:grplsoopt-ns}), scale-invariance considerations have not been taken into account. Usually the individual penalty level $\omega_{j}$'s could be chosen proportional to the scale $\sigma$ as a remedy. This issue has been discussed and studied, pertaining to the Lasso problem, in the literature. See  \cite{Huber11}, \cite{SBG10}, \cite{Antoniacomm10}, \cite{SZcomm10}, 
\cite{BelloniCW11}, \cite{Sun2012}, \cite{SunZ13} and many more. 
For the group Lasso problems, this issue has been tackled via the square-root group Lasso formulation in \cite{BLS14}. Here we follow the the prescription from  \cite{Antoniacomm10} and define an optimization problem,
\begin{align}\label{eq:grplsoopt-s}
& (\hbbeta,\hsigma) = \argmin_{\bbeta,\sigma}\ \calL_{\bomega}(\bbeta,\sigma), \\
& \text{where }\ \calL_{\bomega}(\bbeta,\sigma)  = \dfrac{\norm{\by-\matX\bbeta}{2}^{2}}{2n\sigma} + \dfrac{(1-a)\sigma}{2} + \sum^{M}_{j=1}\omega_{j}\norm{\bbeta_{G_j}}{2}.\label{eq:grplsoopt-s2}
\end{align}
Following \cite{SZcomm10} we define an iterative algorithm for the estimation of 
$\{\bbeta,\sigma\}$, 
\begin{align}\label{algo:iter}
\begin{array}{ccl}
\hsigma^{(k+1)} & \leftarrow & \norm{\by-\matX\hbbeta^{(k)}}{2}/ \sqrt{(1-a)n},\\[0.2cm]
\bomega' & \leftarrow & \hsigma^{(k+1)}\bomega,\\[0.2cm]
\hbbeta^{(k+1)} & \leftarrow &  \argmin_{\bbeta} \calL_{\bomega'}(\bbeta),
\end{array}
\end{align}
where $\calL_{\bomega'}(\bbeta)$ was as defined in (\ref{eq:grplsoopt-ns}). 
Due to the convexity of the joint loss function $\calL_{\bomega}(\bbeta,\sigma)$, 
the solution of (\ref{eq:grplsoopt-s}) and the limit of (\ref{algo:iter}) give the same estimator. 
Moreover, if the minimization of $\sigma$ is first taken with the unknown $\bbeta$ 
in (\ref{eq:grplsoopt-s}), the second minimization of $\min_\sigma \calL_{\bomega}(\bbeta,\sigma)$ 
over $\bbeta$ becomes the square-root group Lasso problem of \cite{BLS14} 
when $\omega_j \propto |G_j|^{1/2}$. 
As the aim of this paper is statistical inference of group effects,   
the formulation in (\ref{eq:grplsoopt-s}) explicitly provides a needed estimate of $\sigma$. 
Moreover, we use a different penalty $\omega_j \propto |G_j|^{1/2} + \sqrt{2\log(M/\delta)}$ to 
benefit from group sparsity in the estimation of both $\bbeta$ and $\sigma$ and in prediction as well. 

The constant $a\geq 0$ provides control over the degrees of freedom adjustments. 
For simplicity, we take $a=0$ for all subsequent discussions. It is clear that that with $a=0$ and $\bomega'=\hsigma\bomega$, one has $\hsigma\calL_{\bomega}(\bbeta,\hsigma) = \calL_{\bomega'}(\bbeta) + \hsigma^{2}/2$. The algorithm in (\ref{algo:iter}) suggests a profile optimization approach. The following lemma is similar to Proposition 1 in \cite{Sun2012} and characterizes the solution via partial derivative of the profile objective.

%
\begin{lemma}\label{lem:partialdelsig}
Let $\hbbeta(\bomega)$ denote a solution of the optimization problem in (\ref{eq:grplsoopt-ns}). 
Then, $\hbbeta(\sigma\bomega)$ is a minimizer of $\calL_{\bomega}(\bbeta,\sigma)$ in 
(\ref{eq:grplsoopt-s2}) for given $\sigma$, and the profile loss function 
$\calL_{\bomega}(\hbbeta(\sigma\bomega),\sigma)$ is convex and continuously differentiable in $\sigma$ with 
\begin{align}
\dfrac{\pa}{\pa\sigma} \calL_{\bomega}(\hbbeta(\sigma\bomega),\sigma) = \dfrac{1}{2} -\dfrac{ \norm{\by-\matX\hbbeta(\sigma\bomega)}{2}^{2}}{2n\sigma^{2}}. 
\end{align} 
Moreover, the algorithm in (\ref{algo:iter}) converges to a minimizer $(\hbbeta,\hsigma)$ 
in (\ref{eq:grplsoopt-s}) satisfying $\hbbeta = \hbbeta(\hsigma\bomega)$, and 
the estimator $\hbbeta$ and $\hsigma$ are scale equivariant in $\by$. 
\end{lemma}

\noindent The proof of Lemma \ref{lem:partialdelsig} is relegated to the \nameref{sec:appendix}. We now present the consistency theorem which extends 
Theorem \ref{th:betal12-ns} by providing convergence results for the estimate of scale.  Define 
\bes
\mu(\bomega, \xi) =  \frac{2\xi \sum_{j \in \Sstar}\omega^{2}_{j}}{\SCIF^{(G)}_1(\xi, \bomega,\Sstar,\Sstar)},\quad  
\tau_- = \frac{2\mu(\bomega,\xi)(\xi-1)}{\xi+1},\quad 
\tau_+ = \frac{\tau_-}{2} + \mu(\bomega,\xi). 
\ees
Let $m_{d,n}$ be the median of the beta$(d/2,n/2-d/2)$ distribution and define 
\bes
\omega_{*,j} \ge \sqrt{m_{d_j,n}} + \sqrt{\frac{2\log(M/\delta)}{(n\vee 2)-3/2}},\ 
A_* = \frac{(\xi+1)/(\xi-1)}{\sqrt{\{1- 2\mu(\bomega_*,\xi)(\xi+1)/(\xi-1)\}_+}}, 
\ees
where $\bomega_*$ is the vector with elements $\omega_{*,j}$ and $d_j=|G_j|$. We will show that 
$ \sqrt{m_{d_j,n}} \le  (d_j/n)^{1/2}+n^{-1/2}$ in the proof of the following theorem.  

\begin{theorem}\label{th:betasgrp}
Let $\{\hbbeta,\hsigma\}$ be a solution of the optimization problem (\ref{eq:grplsoopt-s2}) 
with data $(\matX,\by)$ and  
$\bbeta^*$ be a vector with $\supp(\bbetastar) \subset G_{\Sstar}$ for some 
$\Sstar \subset T^*=\{1,\cdots, M\}$. 
Let $\xi>1$. \\
(i) Suppose  $\SCIF^{(G)}_1(\xi,\bomega,\Sstar,\Sstar) >0$ in (\ref{eq:SCIF}) and $\tau_+<1$. 
Define the following event
\begin{align}\label{calE-s}
\calE = \left\{\max_{1\le j\le M}
\frac{\norm{\matXgj^{T}(\by-\matX\bbetastar)}{2}}{\omega_{j} n\sigstar/\sqrt{1+\tau_-}} 
< \dfrac{\xi-1}{\xi+1}\right\},
\end{align}
where $\sigstar=\norm{\by-\matX\bbetastar}{2}/\sqrt{n}$ is the oracle noise level. 
Then in the event $\calE$, we have
\bel{sgl-sigma}
\frac{\sigstar}{\sqrt{1+\tau_-}}\le \hsigma \le \frac{\sigstar}{\sqrt{1-\tau_+}},
\eel
\bel{sgl-pred}
\norm{\matX\hbbeta-\matX\bbeta^{*}}{2}^2/n 
\le \frac{(\sigstar)^2\{2\xi/(\xi+1)\}^2\sum_{j \in \Sstar}\omega^{2}_{j}}
{(1-\tau_+)\SCIF^{(G)}_1(\xi, \bomega,\Sstar, \Sstar)},  
\eel
and for all $q\ge 1$
\begin{align}\label{sgl-est}
\bigg\{\sum^{M}_{j=1}\omega_{j}^2\bigg(\frac{\norm{\hbbetagj-\bbeta_{\Gj}^{*}}{2}}
{\omega_{j}}\bigg)^q\bigg\}^{1/q}
\leq \dfrac{\sigstar\{2\xi/(\xi+1)\}\big(\sum_{j \in \Sstar}\omega^{2}_{j}\big)^{1/q}}
{\sqrt{1-\tau_+}\SCIF^{(G)}_q(\xi, \bomega, \Sstar,T^*)}. 
\end{align}
(ii) Suppose the regression model in (\ref{eq:mod1}) holds 
with Gaussian error, $\by-\matX\bbeta^*\sim \sfN_{n}(\bzero,\sigma^{2}\matI_{n}\,)$.  
Suppose $\omega_{j} \ge A\|\matX_{G_j}/\sqrt{n}\|_{S}\omega_{*,j} $ with $A\ge A_*$. Then, 
\bel{th-6-5}
\bbP(\calE) \geq 1-\delta
\eel
with the event $\calE$ in (\ref{calE-s}). 
Moreover, if $\sqrt{n}\mu(\bomega,\xi) \rightarrow 0$, then
\begin{align}\label{eq:teststatsigma}
\sqrt{n}\left({\hsigma}/{\sigma} - 1\right) \toD \sfN(0,1/2).
\end{align}
\end{theorem}


Theorem \ref{th:betasgrp}, whose proof is again relegated to the \nameref{sec:appendix}, 
provides explicit rates and constants for mixed $\ell_{q}$ norm estimation of $\bbetastar$ and estimation of scale parameter $\sigma$. 
When $\omega_j\asymp \omega_{*,j}$ 
and $\SCIF^{(G)}_1(\xi, \bomega, \Sstar)\asymp 1$, we have 
\bes
\hbox{$\sum_{j\in T}$}\omega_j^2\asymp \mu(\bomega,\xi)\asymp \big\{s+g\log(M/\delta)\big\}\big/n. 
\ees
It also establishes the veracity of the working assumption in (\ref{eq:prelimbetaconsis}). 
The following Corollary~\ref{cor:grpsparse-ns} provides a more succinct summary to make clear the connection of Theorem \ref{th:betasgrp} to (\ref{eq:prelimbetaconsis}). 

\begin{corollary}[Verification of working assumption for deterministic designs] 
\label{cor:grpsparse-ns} 
Let $\{\hbbeta,\hsigma\}$ be as in (\ref{eq:grplsoopt-s2}) with a penalty 
level satisfying $\omega_{j}/A^* \le \|\matX_{G_j}/\sqrt{n}\|_{S}\omega_{*,j} \le \omega_j/A_*$. 
Suppose the design matrix $\matX$ satisfy the condition 
$\|\matX_{G_j}/\sqrt{n}\|_S^2\le c^*$ and that the sign-restricted cone invertibility condition holds 
in the sense of $\SCIF^{(G)}_q(\xi, \bomega, \Sstar,\Sstar)>c_*$ for some fixed $c_*>0$. 
Suppose $\by-\matX\bbeta^*\sim \sfN_{n}(\bzero,\sigma^{2}\matI_{n}\,)$ and 
$\supp(\bbeta^*)\subseteq G_{S^*}$ with $|G_{\Sstar}|+|\Sstar|\log(M/\delta)\le a_0n$. 
Then, for certain constants $\{a_*,C\}$ depending on $\{c_*,c^*,\xi, A^*\}$ only, 
\bel{cor-1-1}
&& \max\bigg\{\Big|1-\dfrac{\hsigma}{\sigstar}\Big|\,,\,
\frac{\|\matX\hbbeta-\matX\bbeta^*\|_2^2}{n\sigma^2}\,,\,
\sum^{M}_{j=1}\frac{\norm{\hbbetagj-\bbeta_{\Gj}^{*}}{2}}{\sigma/\omega_{j}}\,,\,
\sum^{M}_{j=1}\frac{\norm{\matX_{G_j}(\hbbetagj-\bbeta_{\Gj}^{*})}{2}}{n^{1/2}\sigma/\omega_{j}}\bigg\} 
\cr && \leq C\left\{|G_{\Sstar}|+|\Sstar|\log(M/\delta)\right\}\big/n
\eel
with probability at least $1-\delta$ whenever $a_0\le a_*$. 
\end{corollary}

Corollary \ref{cor:grpsparse-ns} touches upon the mixed prediction loss 
$\sum^{M}_{j=1}\omega_{j}\norm{\matX_{G_j}\hbbetagj-\matX_{G_j}\bbeta_{\Gj}^{*}}{2}$
the first time in this section. The reason for this omission is two fold. Firstly, 
\bes
\bigg\{\sum^{M}_{j=1}\omega_{j}^2\bigg(\frac{\norm{\matX_{G_j}(\hbbetagj-\bbeta_{\Gj}^{*})}{2}}
{n^{1/2}\omega_{j}}\bigg)^q\bigg\}^{1/q}
\le \max_{j\le M}\left\|\frac{\matX_{G_j}}{\sqrt{n}}\right\|_S 
\bigg\{\sum^{M}_{j=1}\omega_{j}^2\bigg(\frac{\norm{\hbbetagj-\bbeta_{\Gj}^{*}}{2}}
{\omega_{j}}\bigg)^q\bigg\}^{1/q}
\ees
so that (\ref{eq:betal12-ns}) and (\ref{sgl-est}) automatically generate the corresponding bounds 
for the mixed prediction error under the respective conditions. Secondly, upper bounds for the 
mixed prediction loss can be obtained by reparametrization within the given group structure as 
in the following corollary. 

\begin{corollary} \label{cor-2} 
Let $\matX_{G_j} = \matU_{G_j}\bLambda_{G_j}\matV_{G_j}^T$ be the SVD of 
$\matX_{G_j}$ with $\bLambda_{G_j}\in\bbR^{|G_j|\times |G_j|}$. 
Define $\bb$ by $\bb_{G_j}=\bLambda_{G_j}\matV_{G_j}^T\bbeta_{G_j}$ 
and $\matU$ by $\matU\bb = \sum_{j=1}^M\matU_{G_j}\bb_{G_j}$. 
Then, 
\bes
\bigg\{\sum^{M}_{j=1}\omega_{j}^2\bigg(\frac{\norm{\matX_{G_j}\hbbetagj-\matX_{G_j}\bbeta_{\Gj}^{*}}{2}}
{\omega_{j}}\bigg)^q\bigg\}^{1/q}
&=& \bigg\{\sum^{M}_{j=1}\omega_{j}^2\bigg(\frac{\norm{\hbb_{G_j}-\bb_{\Gj}^{*}}{2}}
{\omega_{j}}\bigg)^q\bigg\}^{1/q}
\cr  &\leq& \dfrac{2\sigstar \xi\big(\sum_{j \in \Sstar}\omega^{2}_{j}\big)^{1/q}}
{\sqrt{1-\tau_+}\SCIF^{(G)}_q(\xi, \bomega, \Sstar, \Sstar)}
\ees
for all $q\ge 1$ when the conditions for (\ref{sgl-est}), including the definition of the estimator and the SCIF, 
hold with $\matX$, $\bbeta$ and $\bbeta^*$ replaced by $\matU$, 
$\bb$ and $\bb^*$ respectively. 
\end{corollary}

\begin{remark}
Corollary \ref{cor:grpsparse-ns} can be viewed as a scaled version of the main results of 
\cite{HZ10} although here the regularity condition of the design is of a weaker form 
and smaller penalty levels are allowed.  
\end{remark}

\subsection{Random designs}\label{sub-3-3}
In this subsection, we verify the working assumption for sub-Gaussian designs 
by checking the groupwise cone invertibility condition. 
Our analysis also provides lower bounds for the groupwise restricted eigenvalue 
and compatibility constant. 
We first state in the following theorem the main result for random designs. 

\begin{theorem}[Verification of working assumption for random designs] 
\label{th-7} 
Let $0<c_*\le c^*$ and $0<\delta<1 < A_*<A^*$ be fixed constants and $\{\hbbeta,\hsigma\}$ 
be a solution of (\ref{eq:grplsoopt-s2}) with 
\[
\omega_{j}/A^* \le \|\matX_{G_j}\|_S\big\{\sqrt{d_{j}} + \sqrt{2\log(M/\delta)}\big\}\big/n \le \omega_j/A_*. 
\] 
Let $\sigma^* = \|\by-\matX\bbeta^*\|_2/\sqrt{n}$. 
Suppose $\matX$ satisfies the sub-Gaussian condition (\ref{subGaussian-cond}) with 
$c_*\le $eigenvalues$(\bSigma)\le c^*$, 
$\by-\matX\bbeta^*\sim \sfN_{n}(\bzero,\sigma^{2}\matI_{n}\,)$, and 
$\supp(\bbeta^*)\subseteq G_{S^*}$ with 
\bel{th-7-1}
\max_{1\le j\le M}\Big(|G_j|+\log(M/\delta)\Big)I_{\{|S^*|>0\}}
+ |G_{S^*}|+|S^*|\log(M/\delta) \le a_0n. 
\eel
Then, there exist constants $a_*$ and $C$ depending on $\{c_*,c^*,A_*,A^*\}$ only 
such that 
\bel{th-7-2}
&& \max\bigg\{\Big|1-\dfrac{\hsigma}{\sigstar}\Big|,
\frac{\|\matX\hbbeta-\matX\bbeta^*\|_2^2}{n\sigma^2}\,,\,
\sum^{M}_{j=1}\frac{\norm{\hbbetagj-\bbeta_{\Gj}^{*}}{2}}{\sigma/\omega_{j}}\,,\,
\sum^{M}_{j=1}\frac{\norm{\matX_{G_j}(\hbbetagj-\bbeta_{\Gj}^{*})}{2}}{n^{1/2}\sigma/\omega_{j}}\bigg\} 
\cr && \leq C\left\{|G_{\Sstar}|+|\Sstar|\log(M/\delta)\right\}\big/n
\eel
with probability at least $1-\delta$ whenever $a_0\le a_*$. 
\end{theorem}

Theorem \ref{th-7} justifies the working assumption for sub-Gaussian designs. 
It demonstrates the benefit of the strong group sparsity as the sample size 
condition (\ref{th-7-1}) is typically weaker than the usual $\|\bbeta^*\|_0\{1+\log(p/\delta)\} \le a_0n$ 
for the Lasso when $\supp(\bbeta) = G_{S^*}$. 
We omit its proof as it is a direct consequence of Theorem \ref{th:betasgrp} and 
Proposition \ref{prop-2} below. {}{ We preface the presentation of Proposition \ref{prop-2} by first defining the following quantities.}

Let $q>1$ and $\bff = (f_1,\ldots,f_M)^T$ with $f_j>0$. Define 
\bel{rho}
\rho_{q}(s) = \inf_{\bu}\sup_{\bv}\bigg\{\frac{\bv^T(\matX^T\matX/n)\bu}{\|\bv\|_{(q/(q-1))}\|\bu\|_{(q)}}: 
\supp(\bu)=\supp(\bv)=G_{B}, \min_{|B\setminus S|\le 1}\|\bff_S\|_2^2 < s \bigg\}
\eel
with weighted $\ell_{2,q}$ norm 
$\|\bv\|_{(q)} = \Big(\hbox{$\sum_{j=1}^M$} f_j^2\big(\|\bv_{G_j}\|_2/f_j\big)^q\Big)^{1/q}$, 
and 
\bel{theta}
&& \theta_q(s,t)=\sup\bigg\{\frac{\bv^T(\matX^T\matX/n)\bu}
{\|\bv\|_{(q/(q-1))}\|\bu\|_{(q)}}: \supp(\bu)=G_{B_1},\supp(\bv)=G_{B_2},
\cr && \qquad\qquad\qquad\quad |B_k\setminus S_k|\le 1,\|\bff_{S_1}\|_2^2 < s, \|\bff_{S_2}\|_2^2 < t, 
B_1\cap B_2= \emptyset \bigg\}. 
\eel
Under the norm $\|\cdot\|_{(q)}$, 
$1/\rho_q(s)$ is the maximum operator norm of $(\matX_{G_B}^T\matX_{G_B}/n)^{-1}$ in {}{$\bbR^{|G_B|}$}, 
and $\theta(s,t)$ is the maximum operator norm of $\matX_{G_{B_2}}^T\matX_{G_{B_1}}/n$. 
In particular, $\rho_2(s)$ is the smallest eigenvalue of 
$\matX_{G_B}^T\matX_{G_B}/n$ under the given constraints on the support set $G_B$. 
Let $a_q = (1-1/q)/q^{1/(q-1)}$. 
For $\xi>0$, $T\subset\{1,\ldots,M\}$, $t_0 = \sum_{j\in T}f_j^2$, $x_0 \ge 1$, $1\le y_0 \le x_0/a_q$ 
and $m\in \{1,2\}$, define quantities 
$C_q(\xi,x_0,y_0) = \xi +\big(1+a_qy_0-x_0\big)_+x_0^{-1/q}$ and 
\bel{kappa_m}
\kappa_{q,m}(\xi,t_0,x_0,y_0) = \rho_q(x_0t_0) - m\theta_q(x_0t_0,y_0t_0)
y_0^{1/q-1}C_q(\xi,x_0,y_0). 
\eel

\begin{proposition}\label{prop-2} (i) 
Suppose $\omega_j = C_nf_j$ for some constant $C_n$ not depending on $j$. Then, 
\bel{prop-2-1a}
& {\rm RE}^{(G)}(\xi,\bomega,T,T')
\ge \kappa_{2,2}^{1/2}(\xi,t_0,x_0,y_0)/\{1+\delta'\big(1+\xi)/2\},
\\ \label{prop-2-1b}
& {\rm CC}^{(G)}(\xi,\bomega,T) 
\ge \kappa_{2,2}^{1/2}(\xi,t_0,x_0,y_0)
\\ \label{prop-2-2}
&\displaystyle 
{\rm CIF}^{(G)}_q(\xi,\bomega,T,T')
\ge \frac{\kappa_{q,1}(\xi,t_0,x_0,y_0)}{(x_0+\max_jf_j^2/t_0)^{1/q}\{1+\delta'\big(1+\xi)a_q^{1-1/q}\}}, 
\eel
with $\delta' = 0$ for $T'=T$ and $\delta'=1$ for $T'=T^*$, and for $1\le q\le 2$ 
\bel{prop-2-3}
\min\Big({\rm SCIF}^{(G)}_q(\xi,\bomega,T,T'),\frac{{\rm CIF}^{(G)}_q(\xi,\bomega,T,T')}{(1+\xi)^{-1}}\Big)
\ge \frac{\kappa_{2,1}(\xi,t_0,x_0,y_0)}{1+\delta'\big(1+\xi)a_q^{1-1/q}}. 
\eel
(ii) Suppose $\matX$ satisfies the sub-Gaussian condition (\ref{subGaussian-cond}) with 
$c_*\le $eigenvalues$(\bSigma)\le c^*$
and $\omega_j/A^*\le C_n\|\matX_{G_j}/\sqrt{n}\|_S\big\{\sqrt{|G_j|}+\sqrt{2\log(M/\delta)}\big\}
\le \omega_j/A_*$, where $\{c_*,c^*,A_*,A^*\}$ are positive constants. 
Let $\delta' = 0$ for $T'=T$ and $\delta' =1$ for $T'=T^*$.  
For any $\epsilon_0\in (0,1)$, there exists $a_0$ depending on 
$\{\epsilon_0,c_*,c^*,A_*,A^*\}$ only such that 
\bes
\SCIF^{(G)}_q(\xi,\bomega,T,T')
\ge (1-\epsilon_0)\lam_{\min}(\bSigma)\big/\big\{1+\delta'\big(1+\xi)a_q^{1-1/q}\big\},\quad 1\le q\le 2. 
\ees
with at least probability $1-\delta$ whenever (\ref{th-7-1}) holds. 
Moreover, the inequality also holds with $\SCIF^{(G)}_q(\xi,\bomega,T,T')$ replaced 
by $\big\{{\rm RE}^{(G)}(\xi,\bomega,T,T')\big\}^2$ for $q=2$, 
by $\big\{{\rm CC}^{(G)}(\xi,\bomega,T)\big\}^2$ for $q=1$ and $T'=T$, 
or by $(1+\xi){\rm CIF}^{(G)}_q(\xi,\bomega,T,T')$. 
\end{proposition}


\section{Simulation Results}\label{sec:simures}
In this section we provide a few simulation results in support of our theory developed in Sections \ref{sec:groupinf} and \ref{sec:consisres}. As a prelude, we first show the performance of 
the scaled group Lasso procedure in a simulation experiment.

\subsection{Normality of estimate of the scale parameter} We consider two simulation designs with  $(n=1000,p=200)$ and $(n=1000,p=2000)$ design matrices with the elements of the design matrix generated independently from $\sfN(0,1)$. We assume that the true parameter $\bbetastar$ has an inherent grouping with total set of $p$ parameters divided into groups of size $d_{j}=4$. In the design  $(n=1000,p=200)$ we have total number of groups $M=50$ and in $(n=1000,p=2000)$, $M=500$. For both scenarios, the true parameter $\bbetastar$ is assumed to be ($g=2, s=8$) strong group sparse with its non-zero coefficients in $\{-1,1\}$. Both simulation designs have a $\sfN(0,\sigma^{2})$ error added to the true regression model $\matX\bbetastar$ with $\sigma=1$. We also assume that the design matrix is groupwise  orthogonalized 
in the sense of $\matX_{G_j}^T\matX_{G_j}/n = \matI_{G_j\times G_j}$, $j=1,\ldots,M$. 

In estimation of $\sigma$ we employ the scaled group Lasso procedure as shown in (\ref{algo:iter}). The groupwise penalty factors $\omega_{j}$'s are chosen to equal to $\lambda(\sqrt{d_{j}/n}+ \sqrt{(2/n)\log (M)})$ for some fixed $\lambda>0$. The implementation of group Lasso procedure is via the \textsf\textbf{R} package \texttt{grpreg}. 
\begin{figure}[t]
    \centering
    \begin{minipage}{.35\textwidth}
  \centering
  \includegraphics[width=1.1\linewidth, height=0.3\textheight]{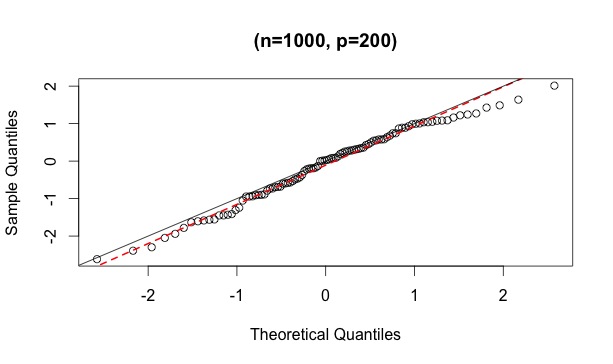}
\end{minipage}
\hspace{0.1\textwidth}%
\begin{minipage}{.35\textwidth}
  \centering
  \includegraphics[width=1.1\linewidth, height=0.3\textheight]{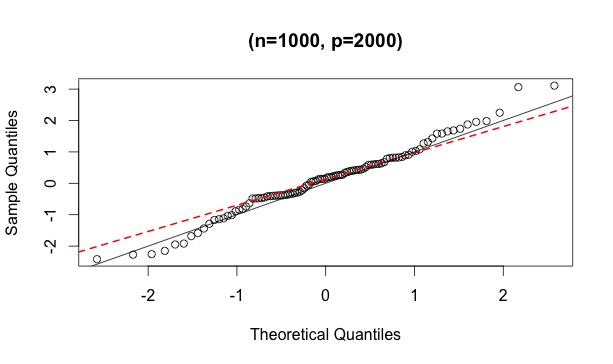}
\end{minipage}%
    \caption{Normal QQ plot for the test statistic for $\hsigma$ in (\ref{eq:teststatsigma}) in Theorem \ref{th:betasgrp} with $n=1000, p=\{200,2000\}, g=2, s=8$. The results are produced with 100 replications of the scaled group Lasso. The red dotted line is fitted through $1^{{\rm st}}$ and $3^{{\rm rd}}$ sample quantiles. }
    \label{fig:qqplotsigma}
\end{figure}

In the design setup with $(n=1000,p=200)$, the estimate of $\hsigma$ averaged over a 100 replications is 0.997 with a standard deviation of 0.02. In the design setup with $(n=1000,p=2000)$, the estimate of $\hsigma$ averaged over a 100 replications is 1.0002 with a standard deviation of 0.02.
Additionally Figure \ref{fig:qqplotsigma} shows the Gaussian QQ plots of the test statistic $\sqrt{2n}\left({\hsigma}/{\sigma} - 1\right)$.

\subsection{Asymptotic distribution of regression parameters}We also seek the empirical validation of the asymptotic convergence of the group $\bbetagj$ as described in our theoretical results. For bias correction we take the penalty function in (\ref{pen-multi-reg}) to be the Frobenius norm and apply group Lasso based optimization.
We also consider a new simulation design which is similar to the earlier design with $(n=1000,p=200)$ and $\sigma=1$. We will consider two different schemes for empirical analysis for asymptotic convergence.
\subsubsection*{Small group sizes}
The true parameter $\bbetastar$  is simulated to be $(s=40, g=10)$ strong group sparse with its nonzero values in the interval [2,3]. More specifically, $\bbetastar$ is grouped into groups of  sizes $d_{j}=4$ for all $j$. We construct the test statistic of $\bmugj$ as in (\ref{test}) for one of the nonzero groups.  The \textbf{left panel} of Figure \ref{fig:param} provides $\chi^{2}_{4}$ based QQ plot for the sample quantiles of our test statistic.

\subsubsection*{Large group sizes}
The true parameter $\bbetastar$  is simulated to be $(s=40, g=2)$ strong group sparse with its nonzero values between [2,3]. More specifically, $\bbetastar$ is grouped into 10 groups each of sizes $d_{j}=20$ for all $j$. We let the sparsity of the true parameter $\bbetastar$ to be $s=40$ contained within 2 separate groups. Again, we construct the test statistic of $\bmugj$ as in (\ref{test}) for one of the nonzero groups. The \textbf{right panel} of Figure \ref{fig:param} shows the QQ plot for this group's size- normalized test statistic as defined in (\ref{eq:largegrp}). As the figure suggests, for large group sizes asymptotic normality of the group test statistic is empirically supported.

\begin{figure}[t]
\centering%
\begin{minipage}{0.35\textwidth}
  \includegraphics[width=1.1\linewidth, height=0.3\textheight]{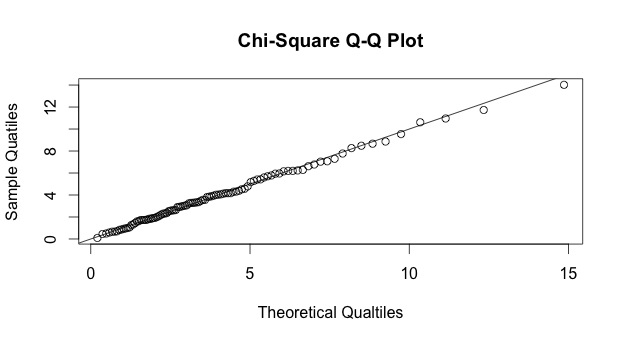}
\end{minipage}%
\hspace{0.1\textwidth}%
\begin{minipage}{0.35\textwidth}\centering
  \includegraphics[width=1.1\linewidth, height=0.3\textheight]{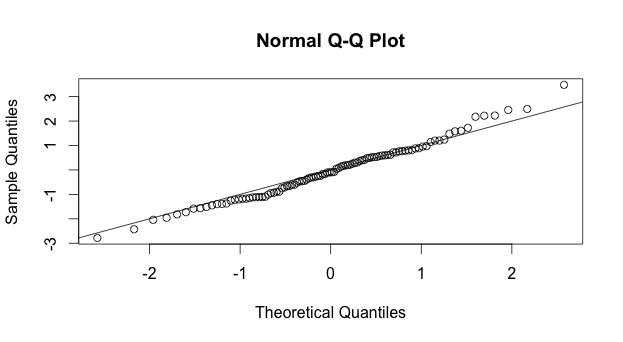}
\end{minipage}%
\caption{The left panel considers test for a \textbf{Small group}. It shows chi-squared QQ plot for the test statistic $T_{G}$ with $n=1000, p=200, g=10, s=40$. The theoretical quantiles were drawn from $\chi^{2}_{4}$ random variable. The group being tested has size 4. The right panel considers test for a \textbf{Large group}. It shows normal QQ plot for the test statistic $(T^{2}_{G}-|G|)/\sqrt{2|G|}$ with $n=1000, p=200, g=2, s=40$.  Here the group size of the test group is 20.}
\label{fig:param}
\end{figure}

\subsection{Comparison with other methods} In this subsection we compare the performance of our group Lasso methods with other recent methods developed for inference in high dimensional models. In particular we consider three different classes of methods.
\begin{table}[ht]
\footnotesize
\begin{center}
\scalebox{0.815}{\begin{tabular}{cccccp{0.005cm}ccccp{0.005cm}ccccp{0.005cm}cc}
\toprule[2pt]
{\multirow{2}{*}{ \textbf{Design}}} &  \multicolumn{4}{c} {\textbf{Proposed Method}} & &\multicolumn{4}{c} {\textbf{Projection Based}} & &\multicolumn{4}{c}{\textbf{Multi sample-split}} & & \multicolumn{2}{c}{\multirow{2}{*}{\textbf{Group Bound}}}\\
\cline{2-5}\cline{7-10} \cline{12-15}\\[-0.7em]
 & \multicolumn{2}{c}{\textbf{Chi-squared}} & \multicolumn{2}{c}{\textbf{Normal}} & &\multicolumn{2}{c}{\textbf{Lasso}}& \multicolumn{2}{c}{\textbf{Ridge}} & & \multicolumn{2}{c}{\textbf{Lasso}} & \multicolumn{2}{c}{\textbf{Group Lasso}} & & \\
 \cline{2-3}\cline{4-5}\cline{7-8}\cline{9-10} \cline{12-13} \cline{14-15}\cline{17-18}\\[-0.7em]
$\boldsymbol{(g,\ s)},\ \boldsymbol{(\rho,\ \tau)}$& \textbf{FP} & \textbf{TP} & \textbf{FP} & \textbf{TP} & & \textbf{FP} & \textbf{TP} & \textbf{FP} & \textbf{TP} & & \textbf{FP} & \textbf{TP} & \textbf{FP} & \textbf{TP}& & \textbf{FP} & \textbf{TP}\\
\hline\\[-0.9em]

(1,\ 5),\ (0,\ 0.1) & 0.04 &0.11 & 0.04 &0.11 &{}& 0 & 0.02 & 0 & 0 & {} & 0 & 0 & 0 & 0 &{} & 0  & 0 \\[0.2em]
(1,\ 5),\ (0,\ 0.5) & 0 & 1 & 0 & 1 &{}& 0 & 1 & 0.01 & 0.2 & {} & 0 & 0.72 & 0 & 0.23 &{} &0  & 0 \\[0.2em]
(1,\ 5),\ (0,\ 1) & 0& 1 & 0 & 1 &{}& 0& 1 & 0 & 1 & {} & 0 & 1 & 0 & 1 &{} &0 & 0 \\[0.2em]
\hline\\[-0.9em]

(1,\ 5),\ (0.5,\ 0.1) & 0.03 & 0.3 & 0.03 & 0.3 &{}& 0 & 0.06 & 0 & 0 & {} & 0 & 0 & 0&0 &{} &0  & 0 \\[0.2em]
(1,\ 5),\ (0.5,\ 0.5) & 0 &1 & 0 & 1 &{}& 0 & 1 & 0 & 0.71 & {} & 0 & 0.99 & 0 & 0.47 &{} &0 & 0.02\\[0.2em]
(1,\ 5),\ (0.5,\ 1) & 0 & 1 & 0 & 1 &{}& 0 & 1 & 0 & 1 & {} & 0 & 1 & 0 & 1 &{} & 0  & 0.97 \\[0.2em]

\hline\\[-0.9em] 
(1,\ 5),\ (0.9,\ 0.1) & 0.02 & 0.45 &0.02 & 0.45 &{}& 0 & 0.02 & 0.2 & 0.02 & {} & 0 & 0.32 & 0 & 0 	&{} &0 & 0 \\[0.2em]
(1,\ 5),\ (0.9,\ 0.5) & 0 & 1 & 0 & 1 &{}& 0 & 1 & 0 & 0.07 & {} & 0 & 0.22 & 0 & 0.01 &{} &0  & 0.12 \\[0.2em]
(1,\ 5),\ (0.9,\ 1) & 0 & 1 & 0 & 1 &{}& 0 & 0.98 & 0 & 0.81 & {} & 0 & 0.86 & 0 & 0.32 &{} &0 & 1\\[0.2em]
	\hline\\[-0.9em]
(1,\ 20),\ (0.9,\ 0.1) & 0 & 1 & 0 & 1 &{}& 0 & 0.38 & 0 & 0.01 &{}&0 &0.05 & 0 & 0.00 &{}& 0 & 0.04 	\\
\bottomrule[2pt]
\end{tabular}}
\end{center}
\caption{Comparison of true positive and false positive rates for three different choices of block correlation $\rho$ and three choices of signal parameter $\tau$. The scale parameter $\sigma=1$ in all cases. The results are based on 100 replications for testing the nonzero group (for TP) and first zero group (FP). Performance of all the tests are good for the strong signal ($\tau$= 1). For the weak signal $\tau=0.1$, group Lasso clearly out-performs other methods.}
\label{tab:comp}
\end{table}
\begin{description}
\item[Projection based:] For the projection based methods, we consider two cases. 1) The Ridge estimation based testing with correction for projection bias that was developed in \cite{Buhl13}. 2) The Lasso relaxed projection followed by bias correction idea developed in \cite{Zhang2014} which is similar to the de-sparsified Lasso in \cite{VandeGeer2013}. These methods are adapted for testing of groups of variables adjustment of individual $p$-values; see \cite{dezeurehdi14}.

\item[Sample split based:] The idea of single sample splitting was developed in \cite{WasRoe09} which involves splitting the sample into two parts. The first part is used to select variables and the second to construct $p$-values for the selected variables in the first model. The final step is to adjust the $p$-values for control of the familywise error rate (FWER). Due to the variability of the $p$-values for different splittings, \cite{mein09} proposed multi sample-splitting idea which involves running the single sample splitting $B$ times and aggregating the $B$ adjusted $p$-values. We employ the multi sample-splitting with two different variable selection procedures: Lasso and group Lasso. For Lasso, the groupwise $p$-value is obtained by  Bonferroni adjustments.

\item[Group bound:] The final procedure we consider is the group bound method developed in \cite{Mein13}. One advantage of this method is that it doesn't require any assumptions on the design matrix.
\end{description}
Implementation of all the above methods are available in the $\textbf{R}$ package \texttt{hdi}; see also \cite{dezeurehdi14}.

\noindent \textbf{Simulation Design:} We consider a very simple simulation design where the design matrix $\matX\in \Re^{n\times p}$ is assumed to have iid rows with each row following $\sfN(\bzero, \bSigma)$, where $\bSigma$ is assumed to be a correlation matrix having a block diagonal structure with block size $k=5$. We take $n=100$ and $p=200$ so that $\bSigma$ has $M=$40 blocks. Within each block, the correlation is assumed to be $\rho$. For our simulations, we consider three possible choices of $\rho$ namely $\{0,0.5,0.9\}$. 

 The true parameter $\bbeta$ is assumed to have the group structure as defined by the block structure of $\matX$. Moreover we assume only the first group has nonzero signals with all of them having the same value $\tau>0$. Thus $\bbetastar$ is of the form,
\[
\bbetastar= (\underbrace{\tau,\tau,\tau,\tau,\tau}_{\text{group 1}}, \underbrace{0,0,0,0,0}_{\text{group 2}},\cdots,\underbrace{0,0,0,0,0}_{\text{group 40}})
\]
Thus in all these cases, the true signal $\bbetastar$ is $(g=1, s=5)$ strong group sparse. We consider three choices of the signal parameter $\tau$: $\{0.1, 0.5, 1\}$.

We also consider an additional scenario, where we take $k=20$ so that number of groups $M=10$ (The last line of Table \ref{tab:comp}). For this case we only compare the performance for signal strength $\tau=0.1$ which highlights the performance of group Lasso.

The responses are simulated by $\by=\matX\bbetastar+\bepsa$ where $\bepsa\sim \sfN(0,\sigma)$. We take the true scale parameter $\sigma=1$ in all simulation designs and estimate $\sigma$ via scaled group Lasso. For application of the group Lasso based testing, we take the group weights equal to $\omega_{j}= 5(\sqrt{d_{j}/n} + \sqrt{(2/n)\log(M)})$ where $M=40$ and $d_{j}=5$ for group sizes=5 and $d_{j}=20$ for group sizes=20.

In Table \ref{tab:comp}, we provide a comparison of the true positive (TP) and false positive (FP) rates for 100 replications. It is clear from the table that group Lasso performs comparably or better than all the other methods. The false positive rates of all the methods are either 0 or close to zero for most of the designs. The true positive (TP) rate (power) of group Lasso method clearly dominates those of the other methods especially when the signal is not strong: $\tau=0.1$. One rationale for this would be the accumulation of small signals in the $\ell_{2}$ norm for the group that is used for the group Lasso. For group bound method, clearly the performance becomes comparable to group Lasso as the blockwise correlation $\rho$ is increased. This phenomenon is also observed for group Lasso procedure to a certain extent.


\section{Summary and Discussion}\label{sec:summdis}
We have considered statistical inference of variable groups in a high-dimensional linear regression setup. 
In particular we show the benefit of grouping in constructing chi-squared-type procedures for group inference. 
We construct such procedures via bias correction and group Lasso based relaxed projection. 
We show the validity of such approximate chi-squared-type inference under sample size conditions 
that could be potentially much weaker than the requirements for Lasso based procedures. 
This particular scaling also offers us valid statistical inference for a group of possibly unbounded 
number of variables. 

A key step of our methodology concerns the nonconvex optimization scheme (\ref{opt}) 
over the set of orthogonal projection matrices. 
To the best of our knowledge, solution of an optimization problem as in (\ref{opt}) is not yet well studied, 
either algorithmically or analytically.
However, we have proposed a convexation of (\ref{opt}) via a multivariate group Lasso with a weighted 
nuclear or Frobenius norm penalty, which provides feasible solutions for the optimization problem. 
As discussed in Remark \ref{remark-1}, our theoretical results only requires feasibility solutions of 
the optimization scheme. 
As the multivariate group Lasso with Frobenius norm penalty can be carried out 
using the group Lasso program, an interesting direction of research would be to 
develop efficient algorithm for the group nuclear norm penalty. 

Since our results can be directly applied to statistical inference for groups of variables with possibly 
unbounded sizes, application of our procedures for sparse nonparametric additive models 
\citep{RLLW09} would be another future direction of research.
%
\appendix
\section{Appendix}\label{sec:appendix}
This appendix provides proof of 

\medskip
{\bf Proof of Proposition \ref{prop-1}.} 
(i) Since both $\matZ_G$ and $\matX_G$ are $n\times |G|$ matrices, 
\bes
|G| = \rank(\matZ_G^T\matX_G)\le \rank(\matX_{G})\wedge \rank(\matZ_{G})\le |G|\wedge n,
\ees 
so that $\rank(\matP_G)=\rank(\matP_G\matX_G)=|G|$ and $\matP_G=\matP_{G,0}$. 
It follows that $\matP_G\matX_G (\matP_G\matX_G)^\dag\matP_G = \matP_G$.  
As $\matZ_G^T\matX_G$ is a $|G|\times|G|$ invertible matrix, 
$\matP_G\matX_G (\matZ_G^T\matX_G)^{-1}\matZ_G^T =  \matP_G$. 
Since $\rank(\matP_G\matX_G)=|G|$, we are allowed to cancel $\matP_G\matX_G$ to 
obtain $(\matP_G\matX_G)^\dag \matP_G =(\matZ_G^T\matX_G)^\dag\matZ_G^T$. 
This proves the first equality in (\ref{prop-1-2}). 
The second equality in (\ref{prop-1-2}) then follows from 
\bes
(\matP_G\matX_G)^\dag \matP_G\big(\matX_G\bbeta_G^* - \matX_G\hbbeta_G^{(init)}\big)
= \bbeta_G^* - \hbbeta_G^{(init)}, 
\ees
(\ref{eq:mod2}) and its estimated version, and the definition of the remainder term. 

(ii) Let $\matZ_1=\matP_G\matZ_G$. 
As $\matP_G=\matP_G\matP_{G,0}$ is the orthogonal projection to $\calR(\matZ_1)$, 
$\matZ_G^T\matX_G = \matZ_G^T\matP_{G,0}\matQ_G\matX_G
=\matZ_1^T\matP_G\matQ_G\matX_G$ and 
$\rank(\matX_G)=\rank(\matP_G\matQ_G)=\rank(\matZ_1^T\matX_G)$, 
so that 
\bes
(\matZ_G^T\matX_G)^\dag
=(\matZ_1^T\matP_G\matQ_G\matX_G)^\dag
=\matX_G^\dag(\matP_G\matQ_G)^\dag (\matZ_1^T)^\dag. 
\ees
Consequently, as $\matQ_G(\matP_G\matQ_G)^\dag 
= (\matP_G\matQ_G)^\dag = (\matP_G\matQ_G)^\dag\matP_G$ and 
$(\matZ_1^T)^\dag\matZ_G^T=\matP_G$, we have 
\bes
\hbmu_G - \hbmu_G^{(init)} 
&=& \matX_G(\matZ_G^T\matX_G)^\dag\matZ_G^T\Big(\by -\matX\hbbeta^{(init)}\Big) 
\cr &=& \matX_G\matX_G^\dag(\matP_G\matQ_G)^\dag (\matZ_1^T)^\dag
\matZ_G^T\Big(\by -\matX\hbbeta^{(init)}\Big) 
\cr &=& (\matP_G\matQ_G)^\dag\matP_G\Big(\by -\matX\hbbeta^{(init)}\Big). 
\ees
This gives (\ref{prop-1-3}). 
As $\matQ_G(\matZ_G^T\matQ_G)^\dag\matZ_G^T = (\matP_G\matQ_G)^\dag\matP_G^T$ 
by the same proof, (\ref{LDPE-G-pred}) also holds. 
Finally, (\ref{prop-1-5}) follows from (\ref{prop-1-1}) and (\ref{eq:mod2}).  

\medskip
{\bf Proof of Lemma \ref{lm-subGaussian}.}
Let $\bu_j, 1\le j\le r_k$, be the eigenvectors of $\matB_k^{T}\bSigma\matB_k$ 
corresponding to positive eigenvalues and $\matU_k=(\bu_1,\ldots,\bu_{r_k})$. 
Let $\matZ_k = \matX\matB_k((\matB_k^{T}\bSigma\matB_k)^\dag)^{1/2}\matU_k\in \bbR^{n\times r_k}$. 
We have $\bbE\matZ_k = {\bf 0}$, $\bbE(\matZ_k^T\matZ_k/n) = \matI_{r_k\times r_k}$, 
$\bbE(\matZ_1^T\matZ_2/n) = \matU_1^T\bOmega_{1,2}\matU_2$, and 
\bes
\sup_{\|\bb\|_2\le 1}\ \bbE \exp\left( \frac{(\bfe_i^T\matZ_k\bb)^2}{v_0} +\frac{1}{v_0} \right) \le 2,\ k=1,2. 
\ees
Moreover, $\matP_k = \matZ_k(\matZ_k^T\matZ_k)^{\dag}\matZ_k^T$ and 
$\|\matU_1^T\bOmega_{1,2}\matU_2\|_S=\|\bOmega_{1,2}\|_S \le 1$. 

For $1\le j\le k\le 2$ and any vectors $\bv_k\in\bbR^{r_k}$ with $\|\bv_k\|_2=1$, 
\bes
\bv_j^T\Big(\matZ_j^T\matZ_k/n - \bbE \matZ_j^T\matZ_k/n\Big)\bv_k 
= \frac{1}{n}\sum_{i=1}^n \Big\{(\bfe_i^T\matZ_j\bv_j)(\bfe_i^T\matZ_k\bv_k) 
- \bv_j^T\bbE(\matZ_j^T\matZ_k/n)\bv_k\Big\}
\ees
is an average of iid variables with 
\bes
&& \bbE \exp\left(\frac{(\bfe_i^T\matZ_j\bv_j)(\bfe_i^T\matZ_k\bv_k) 
- \bv_j^T\bbE(\matZ_j^T\matZ_k/n)\bv_k}{v_0}\right) 
\cr &\le& \left\{\prod_{k=1}^2\sqrt{\bbE \exp\left((\bfe_i^T\matZ_k\bv_k)^2/v_0\right) }\right\}e^{1/v_0} 
\cr &\le& 2. 
\ees
Since the size of an $\epsilon$-net of the unit ball in $\Re^{r_k}$ is bounded by $(1+2/\eps)^{r_k}$,
the Bernstein inequality implies that for $r^*=r_1+r_2$ and a certain numerical constant $C_0$, 
\bes
\bbP\Big\{ \|\matZ_j^T\matZ_k/n - \bbE(\matZ_j^T\matZ_k/n)\|_S > 
C_0v_0\max\Big(\sqrt{t/n+r^*/n},t/n+r^*/n\Big)\Big\}\le e^{-t}/3. 
\ees
This yields (\ref{lm-1-1}) as $\|\matU_1^T\bDelta\matU_2\|_S=\|\bDelta\|_S$ 
for all $\bDelta$ of proper dimension. 

Suppose $\rank(\matP_k)=r_k$. Let $r_0=\rank(\matP_1\matP_2)$ and 
$1\ge \hlam_1\ge \cdots \ge\hlam_{r_0} > 0$ be the (nonzero) singular values of $\matP_1\matP_2$. 
We have $\|\matP_1\matP_2\|_S=\hlam_1$ 
and $\|\matP_1\matP_2^\perp\|_S=\|\matP_1 - \matP_2\|_S = \sqrt{1-\hlam_{\min}^2}$ with 
$\hlam_{\min}=\hlam_{r_0}I\{r_0=r_1=r_2\}$. By definition, 
\bes
\matP_1\matP_2 = \matZ_1(\matZ_1^T\matZ_1)^{-1}\matZ_1^T\matZ_2(\matZ_2^T\matZ_2)^{-1}\matZ_2^T. 
\ees
Since $(\matZ_k^T\matZ_k)^{-1/2}\matZ_k^T$ are unitary maps from the range of $\matP_k$ to $\bbR^{r_k}$, 
the singular values of $\matP_1\matP_2$ is the same as those of 
\bes
(\matZ_1^T\matZ_1)^{-1/2}\matZ_1^T\matZ_2(\matZ_2^T\matZ_2)^{-1/2}. 
\ees

Now suppose that $\|\matZ_j^T\matZ_k/n - \bbE(\matZ_j^T\matZ_k/n)\|_S \le C_0v_0\sqrt{t/n+r/n}\le\eps_0<1$ 
for $1\le j\le k\le 2$. 
Recall that $1\ge\lam_1\ge\cdots\ge \lam_r>0$ are the nonzero singular values of $\bOmega_{1,2}$ 
and $\lam_{\min} = \lam_r I\{r=r_1=r_2\}$.  
As $\bbE(\matZ_k^T\matZ_k/n) = \matI_{r_k\times r_k}$, we have $\rank(\matP_k)=r_k$. 
Moreover, as $\bbE(\matZ_1^T\matZ_2/n) = \matU_1^T\bOmega_{1,2}\matU_2$ with unitary 
maps $\matU_1$ and $\matU_2$, the Weyl inequality implies that 
\bes
\hlam_1\le \frac{\lam_1(1+\eps_0)}{1-\eps_0},\quad \hlam_{\min} \ge \frac{\lam_{\min}(1-\eps_0)}{1+\eps_0}. 
\ees
Thus, (\ref{lm-1-2}) holds.  
As the conditions for $\lam_1<1$ and $\lam_{\min}>0$ follow from the positive-definiteness of $\bSigma$, 
the proof is complete. 

\medskip
{\bf Proof of Theorem \ref{th:betal12-ns}.}
The KKT conditions for the group Lasso asserts that 
\begin{align}\label{eq:kkt-ns}
\begin{array}{lc}
\dfrac{1}{n}\matX_{\Gj}^{T}(\by-\matX\hbbeta) = \omega_{j}\hbbetagj/\norm{\hbbetagj}{2}, & \hbbetagj \neq \bzero,\\[0.3cm]
\dfrac{1}{n}\norm{\matX_{\Gj}^{T}(\by-\matX\hbbeta)}{2} \leq \omega_{j},& \hbbetagj=\bzero.
\end{array}
\end{align}
Let $\bh = \hbbeta-\bbeta^*$. It follows that in the event $\calE$
\begin{align}\label{prf11}
\frac{\norm{\matXgj^{T}\matX\bh}{2}}{\omega_jn}
= \frac{\norm{\matXgj^{T}(\matX\hbbeta-\by+\bepsa)}{2}}{\omega_jn}
\leq 1 + \frac{\norm{\matXgj^{T}\bepsa}{2}}{\omega_jn}
\leq  \frac{2\xi}{\xi+1}.
\end{align}
It also follows from (\ref{eq:kkt-ns}) that in the event $\calE$ 
\bel{pf-th-5-3}
&& \bh_{G_j}^{T}\matX^{T}_{G_j}\matX\bh/n 
\cr &=& \bh_{G_j}^{T}\matX^{T}_{G_j}(\matX\hbbeta-\by+\bepsa)/n 
\cr &\le& \begin{cases} \omega_j\|\bh_{G_j}\|_2 + |\bh_{G_j}^{T}\matX^{T}_{G_j}\bepsa|/n,& j\in S^*,   
\cr - \omega_j\|\bh_{G_j}\|_2 + |\bh_{G_j}^{T}\matX^{T}_{G_j}\bepsa|/n,& j\not\in S^*,  
\end{cases}
\cr &\le& \begin{cases}  \omega_j\|\bh_{G_j}\|_22\xi/(\xi+1),\ j\in S^*,
\cr - \omega_j\|\bh_{G_j}\|_22/(\xi+1),\ j\not\in S^*.
\end{cases} 
\eel
Summing the above inequality over $j$, we have 
\bes
\|\matX\bh\|_2^2/n \le \frac{2\xi}{\xi+1}\sum_{j\in S^*}\omega_j\|\bu_{G_j}\|_2 
- \frac{2}{\xi+1}\sum_{j\not\in S^*}\omega_j\|\bu_{G_j}\|_2. 
\ees
This and (\ref{pf-th-5-3}) implies $\bh \in \scrC^{(G)}_-(\xi, \bomega, \Sstar)$. 
Thus, by (\ref{eq:SCIF}) and (\ref{prf11})
\bes
\norm{\matX\bh}{2}^{2}/n
&\le& \{2\xi/(\xi+1)\}\sum_{j\in \Sstar}\omega_{j}\norm{\hbbetagj-\bbeta_{\Gj}^{*}}{2}
\cr &\leq& \{2\xi/(\xi+1)\}\max_{j}\omega^{-1}_{j}\norm{\matXgj^{T}\matX\bh}{2}
\hbox{$\sum_{j\in S^*}$} \omega_j^2
/\{n\,\SCIF^{(G)}_1(\xi, \bomega, \Sstar,\Sstar)\}
\cr &\leq& \{2\xi/(\xi+1)\}^2\hbox{$\sum_{j\in S^*}$} \omega_j^2
/\{n\,\SCIF^{(G)}_1(\xi, \bomega, \Sstar,\Sstar)\}. 
\ees
Similarly, (\ref{eq:SCIF}) and (\ref{prf11}) yield
\bes
\Big(\hbox{$\sum_{j=1}^M$} \omega_j^2(\|\bh_{G_j}\|_2/\omega_j)^q\Big)^{1/q}
&\le& \big(\hbox{$\sum_{j\in S^*}$} \omega_j^2\big)^{1/q}
\max_{j}\omega^{-1}_{j}\norm{\matXgj^{T}\matX\bh}{2}
/\{\SCIF^{(G)}_q(\xi, \bomega, \Sstar)\}
\cr &\le& \{2\xi/(\xi+1)\}\big(\hbox{$\sum_{j\in S^*}$} \omega_j^2\big)^{1/q}
\{\SCIF^{(G)}_q(\xi, \bomega, \Sstar,T^*)\}. 
\ees

Finally, we prove (\ref{eq:Eprob-ns}).  
Let $\matQ_{G_j}$ be the orthogonal projection to the range of $\matX_{G_j}$. 
As $\bepsa \sim \sfN_{n}(\bzero, \sigma^{2}\matI_{n})$, 
$\|\matQ_{G_j}\bepsa/\sigma\|_2^2\sim \chi^2_{d_j'}$ with $d_j'=\rank(\matQ_{G_j})\le d_j$. 
Thus, it follows from 
the Gaussian concentration inequality that for any $0<\delta <1$, with probability at least $1-\delta$,
\[
\norm{\matXgj^{T}\bepsa}{2}/(\sigma \norm{\matXgj}{S})
\le \norm{\matQ_{G_j}\bepsa/\sigma}{2}
\leq \sqrt{n} \left\{\sqrt{d_{j}} + \sqrt{2\log(1/\delta)}\right\}.
\]
The result in (\ref{eq:Eprob-ns}) follows by an application of the union bound.

\medskip
{\bf Proof of Lemma \ref{lem:partialdelsig}.}
For $\eta \ge 0$ define 
\bes
\calL_{\bomega}(\bbeta,\sigma,\eta) 
= \frac{\|\by-\matX\bbeta\|_2^2}{2n\sigma} + \frac{\sigma}{2}+\sum_{j=1}^M\omega_j\|\bbeta_{G_j}\|_2^{1+\eta}
+ \frac{\eta\sigma^2}{2}
\ees
and $\hbbeta(\sigma\bomega,\eta) = \argmin_{\bbeta}\ \calL_{\bomega}(\bbeta,\sigma,\eta)$. 
As $\calL_{\bomega}(\bbeta,\sigma,\eta)$ is convex in $(\bbeta,\sigma)$, the profile loss 
$\calL_{\bomega}(\hbbeta(\sigma\bomega,\eta),\sigma,\eta)$ is convex in $\sigma$ for all $\eta\ge 0$. 
Note that for $\eta>0$
\bes
&& \dfrac{\pa}{\pa\sigma}\calL_{\bomega}(\hbbeta(\sigma\bomega,\eta),\sigma,\eta) 
\cr &=& \left\{\dfrac{\pa}{\pa\btheta}\calL_{\bomega}(\btheta,\sigma,\eta)\Big{|}
_{\btheta=\hbbeta(\sigma\bomega,\eta)}\right\}^T
\dfrac{\pa\hbbeta(\sigma\bomega,\eta)}{\pa\sigma} + \dfrac{\pa}{\pa t}\calL_{\bomega}
(\hbbeta(\sigma\bomega),t,\eta)\Big|_{t=\sigma}
\cr &=& 1/2 - \|\by-\matX\hbbeta(\sigma\bomega,\eta)\|_2^2/(2n\sigma^2)+\eta\sigma
\ees
as all derivatives involved are continuous. Moreover, as 
$\calL_{\bomega}(\bbeta,\sigma)=\calL_{\bomega}(\bbeta,\sigma,0)$ is strictly convex in $\matX\bbeta$, 
\bes
\lim_{\eta\to 0+}\dfrac{\pa}{\pa\sigma}\calL_{\bomega}(\hbbeta(\sigma\bomega,\eta),\sigma,\eta) 
\to 1/2 - \|\by-\matX\hbbeta(\sigma\bomega)\|_2^2/(2n\sigma^2). 
\ees
Consequently, 
\bes
\calL_{\bomega}(\hbbeta(\sigma_2\bomega),\sigma_2)
- \calL_{\bomega}(\hbbeta(\sigma_1\bomega),\sigma_1)
&=& \lim_{\eta\to 0+}\int_{\sigma_1}^{\sigma_2}
\Big\{\dfrac{\pa}{\pa\sigma}\calL_{\bomega}(\hbbeta(\sigma\bomega,\eta),\sigma,\eta)\Big\} d\sigma
\cr &=& \int_{\sigma_1}^{\sigma_2}\Big\{1/2 - \|\by-\matX\hbbeta(\sigma\bomega)\|_2^2/(2n\sigma^2)\Big\} d\sigma. 
\ees
All other claims follow from the joint convexity of $\calL_{\bomega}(\bbeta,\sigma)$ and 
the strict convexity of the loss function in $\matX\bbeta$.

\medskip
{\bf Proof of Theorem \ref{th:betasgrp}.}
We follow the proof in \cite{Sun2012}.
Let $t\ge \sigma^*/\sqrt{1+\tau_-}$ and $\bh_{\Gj}=\hbbetagj(t\bomega)-\bbetastar_{\Gj}$. 
As the oracle noise level is 
$(\sigstar)^2=\norm{\by-\matX\bbetastar}{2}^{2}/n$, we have 
\bel{pf-1a}
(\sigstar)^2 - \norm{\by-\matX\hbbeta(t\bomega)}{2}^{2}/n 
= (\matX\bh)^{T}(2\bepsa-\matX\bh)/n = (\matX\bh)^{T}(\bepsa + \by-\matX\hbbeta(t\bomega))/n.
\eel
Suppose $\calE$ happens so that $\norm{\matXgj^{T}\bepsa}{2}/n\le t\omega_j(\xi-1)/(\xi+1)$. 
It follows that 
\bes
\left|(\matX\bh)^{T}\bepsa/n\right| = \left|\sum^{M}_{j=1}\bh^{T}_{\Gj}\matXgj^{T}\bepsa/n\right| 
\le \frac{\xi-1}{\xi+1}\sum^{M}_{j=1}t\omega_{j}\norm{\bh_{\Gj}}{2}.
\ees
Moreover, the KKT condition implies 
\[
\left|\bh^{T}\matX^{T}(\by-\matX\hbbeta(t\bomega))/n\right|  
=\left|\sum^{M}_{j=1}\bh_{\Gj}^{T}\matXgj^{T}(\by-\matX\hbbeta(t\bomega))/n\right|
\leq \sum^{M}_{j=1}t\omega_{j}\norm{\bh_{\Gj}}{2}.
\]
As $(\matX\bh)^{T}(2\bepsa-\matX\bh)/n\le 2(\matX\bh)^{T}\bepsa/n$, 
inserting these inequalities to (\ref{pf-1a}) yields
\begin{align*}
-\left(\frac{\xi-1}{\xi+1} + 1\right) \sum^{M}_{j=1}t\omega_{j}\norm{\bh_{\Gj}}{2}
\leq {\sigstar}^{2} - \norm{\by-\matX\hbbeta(t\bomega)}{2}^{2}/n
\leq 2 \frac{\xi-1}{\xi+1}\sum^{M}_{j=1}t\omega_{j}\norm{\bh_{\Gj}}{2}.
\end{align*}
A rescaled version $\hbbeta(t\bomega)$ can be written as 
\bes
\frac{\hbbeta(t\bomega)}{t} = \argmin_{\bb}\left\{\dfrac{\norm{\by/t-\matX\bb}{2}^{2}}{2n} 
 + \sum^{M}_{j=1}\omega_{j}\norm{\bb_{G_j}}{2} \right\}
\ees
as the group Lasso estimator with target $\bbetastar/t$ and noise vector $\bepsa/t$. 
As $t\ge \sigma^*/\sqrt{1+\tau_-}$, the condition of Theorem \ref{th:betal12-ns} is 
satisfied with the rescaled noise $\bepsa/t$, so that  
\[
t^{-1}\sum^{M}_{j=1}\omega_{j}\norm{\bh_{\Gj}}{2} = \sum^{M}_{j=1} \omega_{j}\norm{\hbbeta_{G_j}(t\bomega)/t-\bbetastar_{G_j}/t}{2} < \mu(\bomega,\xi).
\]
As $\tau_- = 2\mu(\bomega,\xi)(\xi-1)/(\xi+1)$ and $\tau_+ = \mu(\bomega,\xi)\{(\xi-1)/(\xi+1)+1\}$, we have 
\begin{align*}
-\tau_+t^2 = 
-\left(\frac{\xi-1}{\xi+1} + 1\right)t^2\mu(\bomega,\xi)
< {\sigstar}^{2} - \norm{\by-\matX\hbbeta(t\bomega)}{2}^{2}/n
< 2 \frac{\xi-1}{\xi+1}t^2\mu(\bomega,\xi) = \tau_-t^2. 
\end{align*}
The upper bound above for $t=\sigma^*/\sqrt{1+\tau_-}$ implies 
\begin{align*}
t^{2} - \norm{y-\matX\hbbeta(t\bomega)}{2}^{2}/n < t^{2} - \sigstar^{2} + \tau_- t^2 = 0, 
\end{align*}
so that $\hsigma > t =\sigma^*/\sqrt{1+\tau_-}$ by Lemma \ref{lem:partialdelsig}. 
Similarly, the lower bound yields $\hsigma < \sigma^*/\sqrt{1-\tau_+}$. 

As $\hsigma > \sigma^*/\sqrt{1+\tau_-}$, the error bounds in Theorem \ref{th:betal12-ns} 
holds for $\{\by/\hsigma,\bbeta^*/\hsigma,\hbbeta/\hsigma\}$, which implies 
(\ref{sgl-pred}) and (\ref{sgl-est}) due to $\hsigma < \sigma^*/\sqrt{1-\tau_+}$. 
When (\ref{eq:mod1}) holds with Gaussian error, 
$|\hsigma/\sigstar-1|=o_P(\mu(\bomega,\xi))=o_P(n^{-1/2})$ 
by (\ref{sgl-sigma}) and the condition on $\mu(\bomega,\xi)$, so that 
(\ref{eq:teststatsigma}) follows from the 
central limit theorem for $\sigstar/\sigma \sim \chi_n/\sqrt{n}$. 
 

It remains to prove (\ref{th-6-5}). Let $\bu^*=\bepsa/\|\bepsa\|_2$, $\matQ_{G_j}$ be the orthogonal 
projection to the range of $\matX_{G_j}$, $d_j'=\rank(\matQ_{G_j})$, and $f(\bu^*)=\|\matQ_{G_j}\bu^*\|_2$. 
As $f(\bu^*)=1$ for $n=1$, we assume $n\ge 2$ without loss of generality. 
The vector $\bu^*$ is uniformly distributed in the sphere $\bbS^{n-1}$ 
and $f(\bu^*)$ is a unit Lipschitz function of $\bu^*$ with median $\sqrt{m_{d_j',n}}\le \sqrt{m_{d_j,n}}$. 
As $\sigma^*=\|\bepsa\|_2/\sqrt{n}$, 
$\|\matX_{G_j}^{T}(\by-\matX\bbeta^*)/(n\sigma^*)\|_2/\|\matX_{G_j}/\sqrt{n}\|_S \le f(\bu^*)$.  
Thus, for $t>0$ and $n\ge 2$, 
\bes
\bbP\left\{\|\bQ_{G_j}\bu^*\|_2\ge \sqrt{m_{d_j,n}}+\frac{t}{\sqrt{n-3/2}}\right\}
\le e^{(4n-6)^{-2}}\bbP\Big\{ \sfN(0,1) > t\Big\}\le e^{-t^2/2}
\ees
by the L\'evy concentration inequality as in Lemma 17 of \cite{SunZ13}. 
It follows that $\bbP(\calE) \geq 1-\delta$ by the union bound when 
$(\xi-1)\omega_j/\{(\xi+1)\sqrt{1+\tau_-}\} \ge \|\matX_{G_j}/\sqrt{n}\|_S\omega_{*,j}$. 
Now, consider $\omega_j=A\|\matX_{G_j}/\sqrt{n}\|_S\omega_{*,j}$. 
Let $\tau_* = 2\mu(\bomega_*,\xi)(\xi-1)/(\xi+1)$. 
It follows from (\ref{eq:cone1}) and (\ref{eq:SCIF}) that $\mu(\bomega, \xi) = A^2\mu(\bomega^*, \xi)$, 
so that $\tau_-=A^2\tau_*$. Consequently, 
\bes
\frac{(\xi-1)\omega_j}{(\xi+1)\sqrt{1+\tau_-}\|\matX_{G_j}/\sqrt{n}\|_S\omega_{*,j}} 
= \frac{(\xi-1)A}{(\xi+1)\sqrt{1+A^2\tau_*}} \ge 1
\ees
if and only if $A \ge \{(\xi+1)/(\xi-1)\}\big/\{1 - \{(\xi+1)/(\xi-1)\}^2 \tau_*\}^{1/2}=A_*$. 
Finally, we note that $\sqrt{m_{d_j,n}} \le \bbE f(\bu^*)+e^{(4n-6)^{-2}}\bbE|\sfN(0,1/(n-3/2))|/2 
\le(d_j/n)^{1/2}+n^{-1/2}$. 

\medskip
{\bf Proof of Proposition \ref{prop-2}.}
(i) We prove that for every $\bu\in \scrC^{(G)}(\xi,\bomega,T)$, 
there exists a non-increasing nonnegative function $h(x)$ and 
$x_0t_0\le t_1 < x_0t_0+\max_j f_j^2$ such that 
\bel{pf-prop-2-0b} 
& \|\bu_{G_T}\|_{(q)}^q = \sum_{j\in T}f_j^{2}(\norm{\bu_{\Gj}}{2}/f_j)^q 
\le \int_0^{t_0}h^q(x)dx,  
\\ \label{pf-prop-2-0c} 
& \|\bu\|_{(q)}^q =\int_0^\infty h^q(x)dx 
\le \big\{1+\big(1+\xi)a_q^{1-1/q}\big\}\big(\int_0^{t_0} h^q(x)dx\big)^{1/q},
\\ \label{pf-prop-2-0d} 
& \hbox{$\max_{j\le M}$}\big[\big\|\matX_{G_j}\matX\bu\big\|_2/(nf_j)\big]t_1^{1/q}
\ge \kappa_{q,1}(\xi,t_0,x_0,y_0)\hbox{$\big(\int_0^{t_1}h^q(x)dx\big)^{1/q}$},
\\ \label{pf-prop-2-0e} 
& \hbox{$\max_{j\le M}$}
\big[\norm{\matXgj^{T}\matX\bu}{2}/(nf_j)\big]\int_0^{t_1}h(x)dx
\ge \kappa_{2,1}(\xi,t_0,x_0,y_0)\int_0^{t_1}h^2(x)dx,
\\ \label{pf-prop-2-0f} & \|\matX\bu\|_2^2/n
\ge \kappa_{2,2}(\xi,t_0,x_0,y_0)\int_0^{t_1}h^2(x)dx. 
\eel
Moreover, for $\bu\in \scrC_-^{(G)}(\xi,\bomega,T)$, 
\bel{pf-prop-2-0g}
& \hbox{$\max_{j\le M}$}
\big[\norm{\matXgj^{T}\matX\bu}{2}/(nf_j)\big]\int_0^{t_0}h(x)dx
\ge \kappa_{2,1}(\xi,t_0,x_0,y_0)\int_0^{t_1}h^2(x)dx. 
\eel

In fact, as $\omega_j\propto f_j$, 
(\ref{prop-2-1a}) and (\ref{prop-2-1b}) follow 
from (\ref{eq:RE}), (\ref{eq:CC}), (\ref{pf-prop-2-0b}), (\ref{pf-prop-2-0c}) and (\ref{pf-prop-2-0f}), 
(\ref{prop-2-2}) follows from (\ref{eq:CIF}), (\ref{pf-prop-2-0b}), (\ref{pf-prop-2-0c}) and (\ref{pf-prop-2-0d}), 
and (\ref{prop-2-3}) follows from (\ref{eq:SCIF}), (\ref{pf-prop-2-0b}), (\ref{pf-prop-2-0c}) and (\ref{pf-prop-2-0g}). 
As these steps of the proof are similar, we only provide the following example: 
\bes
\SCIF^{(G)}_q(\xi,\bomega,T,T^*)
\ge \inf_h\frac{t_0^{1/q}\kappa_{2,1}(\xi,t_0,x_0,y_0)\int_0^{t_1}h^2(x)dx}
{\int_0^{t_0}h(x)dx\big(\int_0^\infty h^q(x)dx\big)^{1/q}}
\ge \frac{\kappa_{2,1}(\xi,t_0,x_0,y_0)}{1+\big(1+\xi)a_q^{1-1/q}} 
\ees
for $1\le q\le 2$ with an application of the H\"older inequality. 

Let us prove (\ref{pf-prop-2-0b})-(\ref{pf-prop-2-0g}) for a fixed $\bu\in \scrC^{(G)}(\xi,\bomega,T)$. 
Relabelling the groups if necessary, we assume without loss of generality 
that $\|\bu_{G_j}\|_2/f_j\ge \|\bu_{G_{j+1}}\|_2/f_{j+1}$ 
for all $1\le j<M$. Let $s_0=0$ and $s_j =\sum_{\ell=1}^j f_\ell^2$ for $1\le j\le M$. 
Define $h(x) = \|\bu_{G_j}\|_2/f_j$ for $s_{j-1}< x\le s_j$, $1\le j\le M$, 
and $h(x)=0$ for $x>s_M$. The identities in (\ref{pf-prop-2-0b}) and (\ref{pf-prop-2-0c}) follow from 
\bel{pf-prop-2-1}
\int_{s_{j-1}}^{s_j} h^q(x)dx = f_j^{2}(\norm{\bu_{\Gj}}{2}/f_j)^q. 
\eel
As $t_0=\sum_{j\in T}f_j^2$ and $h(x)$ is nondecreasing in $(0,\infty)$, 
$\sum_{j\in T}f_j^2(\|\bu_{G_j}\|_2/f_j)^q \le \int_0^{t_0}h^q(x)dx$. 
This gives the inequality in (\ref{pf-prop-2-0b}). 
It follows from (\ref{pf-prop-2-0b}) and the identity in (\ref{pf-prop-2-0c}) that 
$\int_0^\infty h(x)dx \le (1+\xi)\int_0^{t_0}h(x)dx$, 
so that by the shifting inequality 
(\citeauthor{CaiWX10}, \citeyear{CaiWX10}; \citeauthor{YeZ10}, \citeyear{YeZ10}, Eq. (62)) 
\bes
\Big(\int_{t_0}^\infty h^q(x)dx\Big)^{1/q}
\le (a_q/t_0)^{1-1/q}\int_0^\infty h(x)dx
\le \big(1+\xi)(a_q/t_0)^{1-1/q}\int_0^{t_0} h(x)dx. 
\ees
Thus, the inequality in (\ref{pf-prop-2-0c}) follows with an application of the H\"older inequality. 

The proof of (\ref{pf-prop-2-0d}) is a discrete version that of (\ref{pf-prop-2-0c}). 
Let 
\bes
g_1 = \inf\Big\{j \ge 0: s_j \ge x_0t_0\Big\},\ t_1 = s_{g_1},
\ees
with the convention $\inf\emptyset = M+1$, and for $k>1$, 
\bes
g_k = \inf\Big\{j \ge g_{k-1}: s_j \ge t_{k-1}+y_0t_0\Big\},\ t_k = s_{g_k}.
\ees
Recall that $t_0 = \sum_{j\in T}f_j^2$, $x_0 \ge 1$ and $y_0 \le x_0/a_q$. 
It follows from (\ref{pf-prop-2-1}) that 
\bel{pf-prop-2-2}
\sum_{j=1}^{g_k} f_j^{2}(\norm{\bu_{\Gj}}{2}/f_j)^q= \int_0^{t_k}h^q(x)dx,\ k\ge 1. 
\eel
As $h(x)$ is non-increasing in $x$ and $(t_k-t_{k-1})\wedge (t_1/a_q) \ge y_0t_0$, 
another application of the shifting inequality 
(\citeauthor{CaiWX10}, \citeyear{CaiWX10}; \citeauthor{YeZ10}, \citeyear{YeZ10}, Eq. (63)) yields 
\bel{pf-prop-2-6}
&& \sum_{k\ge 2}\Big(\int_{t_{k-1}}^{t_k}h^q(x)dx\Big)^{1/q}
\cr &\le& \sum_{k\ge 2} (y_0t_0)^{1/q-1}
\int_{t_{k-1}-a_qy_0t_0}^{t_k-a_qy_0t_0}h(x\vee t_1)dx
\cr &=& (y_0t_0)^{1/q-1}\int_{t_1-a_qy_0t_0}^\infty h(x\vee t_1)dx
\cr &\le&(y_0t_0)^{1/q-1}\Big(\xi\int_0^{t_0}h(x)dx +\big(t_0-(t_1-a_qy_0t_0)\big)_+h(t_1)\Big)
\cr &\le& \Big(\int_0^{t_1}h^q(x)dx\Big)^{1/q}
\Big(\xi y_0^{1/q-1} +\big(t_0+a_qy_0t_0-t_1\big)_+(y_0t_0)^{1/q-1}t_1^{-1/q}\Big)
\cr &\le & \Big(\int_0^{t_1}h^q(x)dx\Big)^{1/q}
\Big(\xi y_0^{1/q-1} +\big(1+a_qy_0-x_0\big)_+y_0^{1/q-1}x_0^{-1/q}\Big)
\cr &=& \Big(\int_0^{t_1}h^q(x)dx\Big)^{1/q}
\big(\rho_q - \kappa_{q,1}(\xi,t_0,x_0,y_0)\big)\big/\theta_q(x_0t_0,y_0t_0). 
\eel
Let $B_1= \{1,\ldots,g_1\}$ and $B_k= \{g_{k-1}+1,\ldots,g_k\}$ for $k\ge 2$. 
Let 
\bes
\bv = \argmax_{\bw}\Big\{\bw^T\matX^T\matX_{G_{B_1}}\bu_{G_{B_1}}/n: \supp(\bw)\subseteq G_{B_1}, 
\|\bw\|_{(q/(q-1))} =1\Big\}. 
\ees
As $\sum_{j=1}^{g_1-1}f_j^2\le x_0t_0$, it follows from (\ref{rho}) and (\ref{pf-prop-2-2}) that 
\bes
\bv^T\matX^T\matX_{G_{B_1}}\bu_{G_{B_1}}/n
\ge \rho_q(x_0t_0)\|\bu_{G_{B_1}}\|_{(q)} = \Big(\int_0^{t_1}h^q(x)dx\Big)^{1/q}\rho_q(x_0t_0). 
\ees
By (\ref{theta}), 
$\big|\bv^T\matX^T\matX_{G_{B_k}}\bu_{G_{B_k}}\big)\big|
\le \theta_q(x_0t_0,y_0t_0)\big\|\bu_{G_{B_k}}\big\|_{(q)}$, 
so that by (\ref{pf-prop-2-2}) and (\ref{pf-prop-2-6}),
\bes
\bv^T(\matX^T\matX/n)\bu 
&\ge& \Big(\int_0^{t_1}h^q(x)dx\Big)^{1/q}\rho_q(x_0t_0) 
- \sum_{k>1}\theta_q(x_0t_0,y_0t_0)\big\|\bu_{G_{B_k}}\big\|_{(q)} 
\cr &=& \Big(\int_0^{t_1}h^q(x)dx\Big)^{1/q}\rho_q(x_0t_0) 
- \sum_{k>1}\theta_q(x_0t_0,y_0t_0)\Big(\int_{t_{k-1}}^{t_k}h^q(x)dx\Big)^{1/q}
\cr &\ge& \Big(\int_0^{t_1}h^q(x)dx\Big)^{1/q}\kappa_{q,1}(\xi,t_0,x_0,y_0).  
\ees
This yields (\ref{pf-prop-2-0d}) via 
\bes
\bv^T(\matX^T\matX/n)\bu 
&\le& \hbox{$\sum_{j\in B_1}$} f_j^2\big(\|\bv_{G_j}\|_2/f_j\big)
\hbox{$\max_{j\le M}$}\big\|\matX_{G_j}\matX\bu\big\|_2/(nf_j)
\cr &\le& \|\bv\|_{(q/(q-1))} \big(\hbox{$\sum_{j\in B_1}$} f_j^2\big)^{1/q}
\hbox{$\max_{j\le M}$}\big\|\matX_{G_j}\matX\bu\big\|_2/(nf_j)
\cr & = & t_1^{1/q}\hbox{$\max_{j\le M}$}\big\|\matX_{G_j}\matX\bu\big\|_2/(nf_j). 
\ees

For $q=2$, $\rho_q(s)$ is the group-sparse eigenvalue of 
the Gram matrix as explained below (\ref{theta}), so that $\rho_2(s)$ 
is attained with $\bv_{G_B} = \bu_{G_B}/\|\bu_{G_B}\|_2$. 
This gives (\ref{pf-prop-2-0e}) with the following modification of the proof of (\ref{pf-prop-2-0d}): 
\bes
\bv^T(\matX^T\matX/n)\bu 
&\le& \hbox{$\sum_{j\in B_1}$} f_j^2\big(\|\bu_{G_j}\|_2/f_j\big)
\big\{\hbox{$\max_{j\le M}$}\big\|\matX_{G_j}\matX\bu\big\|_2/(nf_j)\big\}\big/\|\bu_{G_{B_1}}\|_2 
\cr &\le& \int_0^{t_1}h(x)dx\Big(\int_0^{t_1}h^2(x)dx\Big)^{-1/2}
\hbox{$\max_{j\le M}$}\big\|\matX_{G_j}\matX\bu\big\|_2/(nf_j). 
\ees
Similarly, (\ref{pf-prop-2-0f}) follows from 
\bes
\|\matX\bu\|_2^2/n 
&\ge& \|\matX_{G_{B_1}}\bu_{G_{B_1}}\|_2^2/n 
+ 2\bu_{G_{B_1}}^T\matX_{G_{B_1}}^T\big(\matX\bu - \matX_{G_{B_1}}\bu_{G_{B_1}}\big)/n
\cr &\ge& \kappa_{2,2}(\xi,t_0,x_0,y_0)\int_0^{t_1}h^2(x)dx. 
\ees

Finally, for $\bu\in \scrC^{(G)}_{-}(\xi,\bomega,T)$, we have (\ref{pf-prop-2-0g}) via (\ref{pf-prop-2-0b}) and 
\bes
(\matX_{G_{B_1}}\bu_{G_{B_1}})^T(\matX\bu)/n
\le \hbox{$\sum_{j\in T}$}  \bu_{G_j}^T\matX_{G_j}\matX\bu/n
\le \hbox{$\sum_{j\in T}$}f_j \|\bu_{G_j}\|_2\max_j\big\|\matX_{G_j}\matX\bu\big\|_2/(nf_j). 
\ees

(ii) Let $f_j = \omega_j/C_n$. Consider the event 
$c_*(1-\eps_0)\le \|\matX_{G_j}/\sqrt{n}\|_S\le (1+\eps_0)c^*$ for all $j\le M$, 
in which $f_j\asymp |G_j|^{1/2}+\sqrt{2\log(M/\delta)}$. 
Let $g^* = \max\big\{|B|: |B\setminus S|\le 2,\|\bff_S\|_2^2 < (x_0\vee y_0)t_0\big\}$ 
be the largest number of groups involved in the definition of $\rho_-(x_0t_0)$ 
and $\theta(x_0t_0,y_0t_0)$, and 
$s^* = \max\big\{|G_B|: |B\setminus S|\le 2,\|\bff_S\|_2^2 < (x_0\vee y_0)t_0\big\}$ 
be the largest number of variables involved. 
As $f_j\asymp |G_j|^{1/2}+\sqrt{2\log(M/\delta)}$ and $(x_0,y_0)$ is fixed, we have 
\bes
s^* + 2g^*\log(M/\delta) \lesssim t_0 + \max_j f_j^2 \lesssim n_*
\ees
with $n_* =\max_{j\le M}\big\{|G_j|+\log(M/\delta)\big\} + |G_T|+|T|\log(M/\delta)$. 

The conclusion follows from part (i) and Lemma \ref{lm-subGaussian}. Let 
\bes
\Omega_n &=& \Big\{c_*(1-\eps_0)\le \|\matX_{G_j}/\sqrt{n}\|_S\le (1+\eps_0)c^*\ \forall j, 
\cr && \qquad \rho_2(x_0t_0) \le (1-\eps_0/2)\lam_{\min}(\bSigma), 
\theta_2(x_0t_0,y_0t_0) \ge (1+\eps_0)c^*\Big\}. 
\ees
Let $\matB_1$ and $\matB_2$ be the orthogonal projections to the subspace of vectors 
$\bv\in\bbR^p$ with support sets $G_{B_1}$ and $G_{B_2}$ respectively, 
$t = (2g^*+2)\log(M/\delta)$ and $\eps_0 = C_1\sqrt{t/n+s^*/n}$ with a sufficiently large $C_1$. 
Since $\{s^* + 2g^*\log(M/\delta)\}/n$ is small for small $a_0$, Lemma \ref{lm-subGaussian} yields 
$\bbP\big\{\Omega_n \big\} \le {M\choose g^*}^2 e^{-t} \le (\delta/M)^2$. 
For $x_0=a_qy_0$ and sufficiently large $y_0$,  
$\kappa_{2,m}(\xi,t_0,x_0,y_0)\ge \rho_2(x_0t_0)(1-\eps_0/2)$ in $\Omega_n$. 
The conclusions of part (ii) then follow from part (i). 

\bibliographystyle{imsart-nameyear}
\bibliography{GroupInf-Revision-20160218}
\end{document}

%% file: riikdef.tex
\newcommand{\bel}{\begin{eqnarray}\label}
\newcommand{\eel}{\end{eqnarray}}
\newcommand{\bes}{\begin{eqnarray*}}
\newcommand{\ees}{\end{eqnarray*}}

\def\convd{\stackrel{{\rm D}}{\longrightarrow}}
\def\toD{\convd}

\def\benu{\begin{enumerate}}
\def\eenu{\end{enumerate}}

\def\argmax{\mathop{\rm arg\, max}}
\def\argmin{\mathop{\rm arg\, min}}

\def\Re{{\mathbb{R}}}

\def\complex{\mathop{{\rm I}\kern-.58em\hbox{\rm C}}\nolimits}
\def\pa{\partial}

\def\rank{\hbox{rank}}
\def\Rem{\hbox{Rem}}

\def\trace{\hbox{trace}}

\def\Var{\hbox{Var}}

\def\supp{\hbox{supp}}

\def\mathbold{\boldsymbol} 

\def\ba{\mathbold{a}}

\def\matA{{\rm \textbf{A}}}

\def\bb{\mathbold{b}}
\def\hbb{{\widehat{\bb}}}

\def\matB{{\rm \textbf{B}}}

\def\scrC{{\mathscr C}}

\def\scrC{{\mathscr C}}

\def\bfe{\mathbold{e}}

\def\calE{{\cal E}}

\def\bbE{\mathbb{E}}

\def\bff{\mathbold{f}}

\def\bh{\mathbold{h}}

\def\hbar{\overline{h}}

\def\matI{{\rm \textbf{I}}}

\def\calL{{\cal L}}

\def\bm{\mathbold{m}}
\def\hbm{{\widehat{\bm}}}

\def\sfN{\mathsf{N}}

\def\calO{{\cal O}}

\def\bP{\mathbold{P}}

\def\bbP{\mathbb{P}}
\def\matP{{\rm \textbf{P}}}

\def\bQ{\mathbold{Q}}

\def\matQ{{\rm \textbf{Q}}}

\def\calR{{\cal R}}
\def\bbR{\mathbb{R}}

\def\bbS{\mathbb{S}}

\def\bu{\mathbold{u}}

\def\matU{{\rm \textbf{U}}}

\def\bv{\mathbold{v}}

\def\matV{{\rm \textbf{V}}}

\def\bw{\mathbold{w}}

\def\bx{\mathbold{x}}

\def\matX{{\rm \textbf{X}}}

\def\by{\mathbold{y}}

\def\bz{\mathbold{z}}
\def\tbz{{\widetilde{\bz}}}
\def\bZ{\mathbold{Z}}

\def\matZ{{\rm \textbf{Z}}}

\def\bzero{\mathbold{0}}

\def\bbeta{\mathbold{\beta}}\def\hbeta{\widehat{\beta}}

\def\hbbeta{{\widehat{\bbeta}}}

\def\bgamma{\mathbold{\gamma}}\def\hgamma{\widehat{\gamma}}

\def\bGamma{\mathbold{\Gamma}}\def\hbGamma{{\widehat{\bGamma}}}

\def\bDelta{\mathbold{\Delta}}

\def\vepsa{\varepsilon}\def\eps{\epsilon}
\def\bepsa{\mathbold{\vepsa}}

\def\btheta{\mathbold{\theta}}

\def\lam{\lambda}
\def\hlam{\widehat{\lam}}

\def\bLambda{\mathbold{\Lambda}}

\def\bmu{\mathbold{\mu}}

\def\hbmu{{\widehat{\bmu}}}

\def\hsigma{\widehat{\sigma}}

\def\bSigma{\mathbold{\Sigma}}\def\hbSigma{{\widehat{\bSigma}}}

\def\bomega{\mathbold{\omega}}

\def\bOmega{\mathbold{\Omega}}

\newcommand{\norm}[2]{\|#1\|_{#2}}

\renewcommand{\qed}{\hfill \mbox{\raggedright \rule{0.1in}{0.1in}}}